\documentclass[12pt]{article}
\usepackage[backref=page,colorlinks,bookmarksopen=true]{hyperref}
\usepackage{pdfsync}
\usepackage{amssymb}
\usepackage{latexsym}
\usepackage{amsmath}
\usepackage{amsthm}
\usepackage[all]{xy}
\usepackage[english]{babel}
\newtheorem{theorem}{Theorem}[section]
\newtheorem{lemma}[theorem]{Lemma}
\newtheorem{proposition}[theorem]{Proposition}
\newtheorem{corollary}[theorem]{Corollary}
\newtheorem{definition}[theorem]{Definition}
\newtheorem{example}[theorem]{Example}

\newtheorem{remark}[theorem]{Remark}

\begin{document}
\title{Fusion modules and amenability of coideals of compact and discrete quantum groups}
\author{Benjamin Anderson-Sackaney and Leonid Vainerman}
\date{July 2, 2023}
\maketitle
\begin{abstract}
    We give a definition of an amenable fusion module over a fusion algebra. A notion of relative integrability for the `coduals' of coideals of compact quantum groups was recently introduced in the joint work of de Commer and Dzokou Talla. We use this property to construct an analogue of the quasi-regular representation. Then, we characterize a certain coamenability property of quasi-regular representations with amenability of their associated fusion modules. Afterwards, we obtain a duality result that generalizes Tomatsu's theorem for this coamenability property and an amenability property of their `codual' coideals (under an additional assumption). As an example, we apply this result to show the fusion modules associated to certain non-standard Podle\'s spheres are amenable.
    
\end{abstract}
\begin{section} {Introduction}

In \cite{HI98} the authors introduced the notion of a {\it fusion algebra} $(R,I^R, N^w_{u,v}, d^R)$, where $R$ is
both a unital ring and a free $\mathbb Z$-module with basis $I^R$ and structural contants $N^{\alpha}_{\zeta,\eta}\in \mathbb Z^+ $ satisfying some natural conditions including the existence of so-called dimension function $d^R$. They also gave the definition of amenability for these objects which has important applications to the study of amenability questions in the subfactor theory and in the theory of compact and discrete quantum groups (in what follows, respectively CQG and DQG). In particular, it was shown in \cite{K08} that
a CQG is coamenable if and only if the fusion algebra generated by its representations is amenable.

Here we introduce the notion of a {\it fusion module} $((M, I^M, c^\beta_{u,\alpha}), d^M)$, where $M$ is an $R$-module and also a free $\mathbb Z$-module with basis $I^M$ and structural constants $c^\beta_{u,\alpha}\in \mathbb Z^+$, $d^M$ being the dimension function satisfying some compatibility conditions with the structure of a fusion algebra in $R$. We also give the definition of amenability for fusion modules which extends the one for fusion algebras and apply it to the study of (co)amenability questions for quantum subgroups and coideals of CQG and DQG.

Fusion algebras often appear as Grothendieck rings of certain $C^*$-tensor categories, for instance, bimodule categories in the subfactor theory or unitary representation categories of CQG. Then a $C^*$-tensor category is
called amenable if the corresponding fusion algebra is amenable (cf. \cite{HY00}). Respectively, fusion modules often appear as $\mathbb Z$-modules corresponding to $C^*$-module categories over $C^*$-tensor categories (see the definitions in Preliminaries), then this module category is called amenable if the corresponding fusion module
is amenable

In our study of the (co)amenability properties of coideals of CQG and DQG one of the main tools is the {\it categorical duality}. It assigns to any coaction $\delta$ of a CQG $G$ represented by its Hopf $*$-algebra $(\mathcal O(G),\Delta)$ on a unital $C^*$-algebra $A$ a unique (up to some equivalence) $C^*$-module category
over $Rep(G)$ (see \cite{dCY13, N14}). This module category is in fact the $G$-equivariant $K_0$ group of $A$ or, equivalently, the $K_0$-group of the crossed product algebra $A\rtimes_\delta G$. The last algebra is isomorphic
to $\underset{i\in I}\oplus K(H_i)$, where $K(H_i)$ are compact operators on some Hilbert spaces and $I$ is some index set \cite[Theorem 19]{B95}. Then one can construct a fusion module with a basis $I$. In particular, this works for any right coideal $*$-subalgebra $A\in \mathcal O(G)$ because it carries a coaction $\delta:A \to A\otimes \mathcal O(G)$ which is the restriction of $\Delta$.

If $H$ is a closed quantum subgroup of a CQG $G$ in the sense that there is another CQG $H$ and a Hopf $*$-algebra surjective homomorphism $q_H:\mathcal O(G)\to \mathcal O(H)$, the set $\mathcal O(H\backslash G)=\{a\in \mathcal O(G)|(q_H\otimes id)\Delta(a)=1\otimes a\}$ is so-called quotient type coideal. In this case the above mentioned set $I$ can be identified with the set $Irr(H)$ of (equivalence classes of) irreducible unitary representations of $H$ \cite[Proposition 4.15]{T08}. Then we show that this fusion module equipped with the classical dimension function is amenable in the sense of our definition if and only if $H$ is coamenable.

An important observation is that any coideal of a CQG $G$ is the form $\mathcal O(H\backslash G)$, where $\mathcal O(H)$ is not necessarily a CQG but only a coalgebra which is also a left $\mathcal O(G)$-module equipped with a left $\mathcal O(G)$-module coalgebra map $q_H:\mathcal O(G)\to\mathcal O(H)$ \cite{dCDT22, Ch18}. Using the reasonings close to the ones in \cite{V05}, we are able to show the Morita equivalence of von Neumann algebras generated by the crossed product $\mathcal O(H\backslash G)\rtimes_\delta G$ and by the algebra dual to the coalgebra $\mathcal O(H)$ (moreover, it is a $*$-algebra). This allows us to assign to any coideal of a CQG
$G$ a fusion module $\mathbb C(Rep(H))$ with basis $Irr(H)$ formed by irreducible $*$-representations of this $*$-algebra (they are finite-dimensional) and the classical dimension function $d^H$.

Next we formulate the property of coamenability of $\mathcal O(H)$ and relate it with the amenability of the above mentioned fusion module. We are able to do that for the class of coideals introduced in \cite{dCDT22} and containing all known at this moment examples. For any such coideal there exists a character $g$ on $\mathcal O(G)$ for which the stabilizer coideal $\mathcal O(H\backslash G)^\bot$ inside the dual discrete multiplier Hopf algebra $\mathcal O(\hat G)$ admits a $g$-invariant integral. We will then say that $l^\infty(\hat H)$ is $g$-integrable. In particular, if $g=1$, $\mathcal O(H\backslash G)$ is a so-called {\it compact quasi-subgroup} - see a discussion in \cite{AS23}.  All coideals of quotient type belong to this class.


In the present paper we construct for coideals $\mathcal{O}(H\backslash G)$ such that $l^\infty(\hat H)$ is $g$-integrable the analog of the Fourier transform for a CQG $G$, and with it, a {\it quasi-regular} representation which allows us to formulate the notion of $g$-coamenability of $H$. This notion generalizes coamenability of $G$ as a CQG. In the classical situation it generalizes the notion of coamenability of a subgroup $\Lambda\leq \Gamma$ of a discrete group $\Gamma$ as in \cite{MP03}, where it characterizes coamenability of the inclusion of certain von Neumann algebras coming from the inclusion of groups $\Lambda\subseteq \Gamma$ \cite{MP03} and has been a point of interest for other operator algebraic and representation theoretic considerations (see, for example, \cite{BK21}).

Here is our main result:
\begin{theorem}\label{Coamenability and Amenable Fusion Modules}
Let $G$ be a CQG and $\mathcal{O}(H\backslash G)$ a coideal such that $l^\infty(\hat H)$ is $g$-integrable. Then $H$ is $g$-coamenable if and only if $((\mathbb{C}[Rep(H)], Irr(H), c^\beta_{u,\alpha}, d^H)$ is amenable as a fusion module over $((\mathbb{C}[Rep(G)], Irr(G), N^w_{u,v}), d^G)$.
\end{theorem}

The most important points of the proof are: (1) an adaptation of Kesten's criterion for arbitrary $*$-representations of 
$\mathcal{O}(G)$ and (2) the identification of $Irr(H)$ with the basis of a fusion module associated to the coideal $\mathcal{O}(H\backslash G)$. The rest of the proof involves adapting constructions in \cite{K08}.

After this, we introduce amenability of $\hat H$ - where $\hat H$ is coming from a coideal $l^\infty(\hat H)$ of $\hat G$ associated to $H$ - which generalizes amenability of $\hat G$ as a DQG. Unlike coamenability, this property is readily defined for arbitrary coideals. It is formulated in terms of the structure of the DQG $\hat G$ and so we view it as a dual property of coamenability of $H$. We justify this view with what follows. With these properties in hand, we prove a partial generalization of Tomatsu's celebrate d theorem \cite{T06} which states that a CQG $G$ is coamenable if and only if its dual $\hat G$ is amenable. For this ``duality theorem'', we introduce the Kac property of $H$, which generalizes the Kac property of $G$ as a CQG.
\begin{theorem}\label{Duality for Coamenability of Quasi Regular Representations}
 Let $G$ be a CQG and let $\mathcal{O}(G/H)$ be a coideal such that $l^\infty(\hat H)$ is $g$-integrable and $H$ is Kac type. Then $\hat H$ is amenable if and only if $H$ is $g$-coamenable.
\end{theorem}
The proof uses an adaptation of Blanchard and Vaes' proof \cite{BV02} of Tomatsu's theorem. An important point to note is that the Kac property of $H$ is only used to prove that amenability of $\hat H$ implies coamenability of $H$. In particular, it is true that coamenability of $H$ implies amenability of $\hat H$ when $\mathcal{O}(G/H)$ is a coideal such that $l^\infty(\hat H)$ is $g$-integrable with no additional assumptions. A noteworthy application of Theorem~\ref{Duality for Coamenability of Quasi Regular Representations} is that we can use it to prove certain non-standard Podle\`s spheres, written in the form $\mathcal{O}(B_t\backslash SU_q(2))$, have that $B_t$ is $\delta_{SU_q(2)}^{-1/2}$-coamenable where $\delta_{SU_q(2)}$ is the modular element for the discrete quantum group $\widehat{SU_q(2)}$ (see Example~\ref{Podles Spheres are Amenable}).

{\bf Acknowledgements} The authors are grateful to Kenny De Commer for useful discussions and for proposing the idea of the proofs of Propositions~\ref{gQuasiRegular Representation Proposition}  and \ref{diag}. The first author was supported by the ANR project ANR-19-CE40-0002.

{\bf Data Availability Statement:} There is no data involved with this work.

\end{section}

\begin{section} {Preliminaries}

\begin{subsection} {$C^*$-tensor categories and module categories}

Our main reference for tensor categories is \cite{EGNO15} and for $C^*$-tensor categories - \cite{NT13}.

A $C^*$-category $\mathcal C$ is a linear category over the complex number field $\mathbb C$, endowed with
Banach space norms on the morphism sets $\mathcal C(X,Y)$ and a conjugate linear anti-multiplicative involution
$\mathcal C(X, Y ) \to \mathcal C(Y, X), T \mapsto T^*$ satisfying the identity $||T^*T|| = ||T||^2$. We assume
that a $C^*$-category is closed under taking subobjects, so that any projection in the $C^*$-algebra $\mathcal C(X)
=\mathcal C(X, X)$ corresponds to a subobject of $X$. We also assume that $\mathcal C$ is closed under finite direct
sums. A $C^*$-category is called semisimple if its morphism sets are finite dimensional. In such categories one can
decompose any object $X$ into a direct sum of simple objects using minimal projections in the finite-dimensional $C^*$-algebra $\mathcal C(X)$. A unitary functor between $C^*$-categories is a linear functor $F$
such that $F(T^*) = F(T)^*$.

A $C^*$-tensor category is a $C^*$-category $\mathcal C$ endowed with a bifunctor $\otimes : \mathcal C \times\mathcal C  \to \mathcal C$, a distinguished object ${\bf 1}$, and natural unitary isomorphisms
${\bf 1}\otimes X \to X$, $X \otimes {\bf 1}\to X$, $F: (X \otimes Y ) \otimes Z \to X \otimes (Y \otimes Z)$
satisfying the standard set of axioms for monoidal categories. We will assume ${\bf 1}$ to be simple, so $\mathcal C({\bf 1})=\mathbb C$. A unitary tensor functor between $C^*$-tensor categories is a unitary functor F together with a unitary isomorphism $F_1 : {\bf 1} \mapsto F({\bf 1})$ and natural unitary isomorphisms $F_2 : F(X) \otimes F(Y ) \mapsto F(X \otimes Y )$ satisfying the standard compatibility conditions.

\begin{definition} \label{mcat} (c.f. \cite{dCY13}) Let $\mathcal C$ be a $C^*$-multitensor category with unit object ${\bf 1}$. A $C^*$-category
$\mathcal M$ is called a left $\mathcal C$-module $C^*$-category if there is a bilinear $*$-functor $\boxtimes:\mathcal C\times\mathcal
M\to\mathcal M$ with natural unitary transformations $(X\otimes Y)\boxtimes M\to X\boxtimes (Y\boxtimes M)$ and ${\bf 1}\boxtimes M\to M\
(X,Y\in \mathcal C, M\in \mathcal M)$ making $\mathcal M$ a left module category over $\mathcal C$ - see \cite{EGNO15}, Chapter 7. If $\mathcal C$
is strict, we say that $\mathcal M$ is strict (resp., indecomposable) if these natural transformations are identities (resp., if, for all
non-zero $M,N\in \mathcal M$, there is $X\in\mathcal C$ such that $\mathcal M(X\boxtimes M,N)\neq 0$).

We say that an object $M\in\mathcal M$ generates $\mathcal M$ if any object of $\mathcal M$ is isomorphic to a subobject of $X\boxtimes M$
for some $X\in\mathcal C$. $\mathcal M$ is said to be semisimple if the underlying $C^*$-category is semisimple.
\end{definition}

One naturally defines a morphism $F:\mathcal M_1\to\mathcal M_2$ between two $\mathcal C$-module $C^*$-categories as a morphism of the
underlying $C^*$-categories equipped with a unitary natural equivalence $F(X\boxtimes M)\to X\boxtimes F(M),\ \forall\ X\in\mathcal C,\
 M\in\mathcal M$ satisfying some coherence conditions (see \cite{dCY13}).

A $C^*$-tensor category $\mathcal C$ is said to be rigid if every its object $X$ has a dual. In particular, if $\mathcal C$ is strict, this means that there is an object $\overline X$ in $\mathcal C$ and morphisms $R \in \mathcal C({\bf 1}, \overline X \otimes X)$ and $\overline R \in \mathcal C({\bf 1}, X \otimes \overline X )$ satisfying the conjugate equations
$$
(id_{\overline X} \otimes \overline R^*)(R \otimes id_{\overline X}) = id_{\overline X},
(id_X \otimes R^*)(\overline R\otimes id_X) = id_X.
$$
Rigid $C^*$-tensor categories with simple units are always semisimple.
A rigid $C^*$-tensor category has a notion of intrinsic, or quantum, dimension, defined by
$$
d^{\mathcal C}(X) = \underset {(R,\overline R)} {min} \{||R||,||\overline R||\},
$$
where $(R, \overline R)$ runs over the solutions of the conjugate equations for $X$.
\end{subsection}
\begin{subsection} {Compact and discrete quantum groups, their subgroups and coideals}

Our main references concerning the theory of CQGs and DQGs and their representations is \cite{NT13, W98}. A {\bf compact quantum group (CQG)} $G$ has an associated integrable unital Hopf $*$-algebra $(\mathcal{O}(G), \Delta_G, \epsilon_G, S_G, h_G)$. Here
\begin{itemize}
    \item $\Delta_G : \mathcal{O}(G)\to \mathcal{O}(G)\otimes \mathcal{O}(G)$ is the coproduct;
    \item $\epsilon_G : \mathcal{O}(G)\to \mathbb{C}$ is the counit;
    \item $S_G : \mathcal{O}(G)\to \mathcal{O}(G)$ is the antipode;
    \item $h_G : \mathcal{O}(G)\to \mathbb{C}$ is the Haar state.
\end{itemize}
Here $\Delta_G$ and $\epsilon_G$ are unital $*$-homomorphisms, $S_G$ is an antiautomorphism of $\mathcal{O}(G)$, and $h_G$ is a state, i.e., a unital positive functional. Rather than give an axiomatic description of the Hopf $*$-algebra structure, we will give concrete descriptions of what they do when applied to unitary representations of $G$.

A unitary finite dimensional representation of $G$ corresponds to a unitary matrix $u\in M_{n_u}(\mathcal{O}(G))$  such that:
\begin{itemize}
    \item $\Delta_G(u_{i,j}) = \sum^{n_u}_{k=1}u_{i,k}\otimes u_{k,i}$
    \item $\epsilon_G(u_{i,j}) = \delta_{i,j}$;
    \item $S_G(u_{i,j}) = u_{j,i}^*$;
\end{itemize}
when we write $u = [u_{i,j}]$ as a matrix. It is shown in \cite{NT13} that the category $\mathcal C=Rep(G)$ of unitary finite dimensional representations of a CQG $G$ is a rigid $C^*$-tensor category. The tensor product is
$$u\cdot v = u_{12}v_{13} \in \mathcal{O}(G)\otimes M_{n_u}\otimes M_{n_v}$$
where we are using the standard leg numbering notation $u_{12} = u\otimes 1$, $u_{23} = 1\otimes u$, $u_{13} = (\Sigma\otimes id)u_{23}(\Sigma\otimes id)$, etc, and where $\Sigma : a\otimes b\mapsto b\otimes a$ is the flip map. We denote the irreducibles by $Irr(G)$. The irreducibles satisfy
$$h_G(u_{i,j}) = \begin{cases} 0 & \text{if} ~ u\neq 1\\ 1 &\text{if} ~ u = 1 \end{cases}, u\in Irr(G).$$
It turns out the matrix coefficients of the irreducible unitary representations of $G$ span the underlying Hopf $*$-algebra:
$$\mathcal{O}(G) = span\{u_{i,j} : 1\leq i,j \leq n_u, u\in Irr(G)\}.$$
From these formulas it is clear that
\begin{itemize}
    \item $(\Delta_G\otimes id)\Delta_G = (id\otimes \Delta_G)\Delta_G$, in other words $\Delta_G$ is coassociative;
    \item $(\epsilon_G\otimes id)\Delta_G = id = (id\otimes \epsilon_G)\Delta_G$;
    \item $(h_G\otimes id)\Delta_G = h_G = (id\otimes h_G)\Delta_G$, in other words, $h_G$ is (left and right) $\hat G$-invariant.
\end{itemize}
The dual space of any coalgebra has an algebra structure via convolution. In particular, the space of linear functionals $\mathcal{O}(G)^*$ on $\mathcal{O}(G)$ is an algebra via
$$\mu*\nu := (\mu\otimes \nu)\circ\Delta_G, \mu, \nu\in \mathcal{O}(G).$$
A unitary representation $u \in M_u\otimes \mathcal{O}(G)$ defines a representation $\pi_u  : \mathcal{O}(G)^* \to M_{n_u}$ by $\pi_u(\mu) = (\mu\otimes id)(u)$. Then there is an algebra isomorphism
$$c_{00}(\hat G) := \oplus_{u\in Irr(G)}M_{n_u} \cong \{h_G(a\cdot) : a\in \mathcal{O}(G)\}.$$
The $*$-algebra $c_{00}(\hat G)$ has an integrable multiplier Hopf $*$-algebra structure
$$(c_{00}(\hat G), \Delta_{\hat G}, \epsilon_{\hat G}, S_{\hat G}, \psi_L, \psi_R).$$
We will not discuss the details here, however, $\hat G$ corresponds to a {\bf discrete quantum group (DQG)} which is the Pontryagin dual of $G$. We denote the von Neumann algebra
$$l^\infty(\hat G) := c_{00}(\hat G)'' = \bigoplus_{u\in Irr(G)}^{l^\infty}M_{n_u}.$$
Then, $\pi_u$ extends to a normal irreducible unital $*$-representation $\pi_u : l^\infty(\hat G)\to M_{n_u}$. Denote $Rep(l^\infty(\hat G))$ as the normal finite dimensional unital $*$-representations on $l^\infty(\hat G)$. The elements of $Rep(l^\infty(\hat G))$ are in bijection with the elements of $Rep(G)$. Let $W =\oplus_{u\in Irr(G)}u$. The coproduct extends to a normal unital $*$-homomorphism $\Delta_{\hat G} : l^\infty(\hat G)\to l^\infty(\hat G)\otimes l^\infty(\hat G)$ and satisfies
$$(id\otimes \Delta_{\hat G})(W) = W_{12}W_{13}.$$
Then, by setting $\pi_u\cdot \pi_v = (\pi_u\otimes \pi_v)\circ\Delta_{\hat G}$, it follows from the definitions that $\pi_v\cdot \pi_u = \pi_{u\cdot v}$. In this way, $Rep(l^\infty(\hat G))$ identifies with $Rep(G)$ as a rigid $C^*$-tensor category.

There is also an operator algebraic perspective on CQGs, which is obtained by taking certain operator algebraic completions of $\mathcal{O}(G)$. Given a unital $*$-representation $\pi : \mathcal{O}(G) \to B(H_\pi)$, we let
$$C_\pi(G) = \overline{\mathcal{O}(G)} \subseteq B(H_\pi).$$
In general, $C_\pi(G)$ is just a unital $C^*$-algebra and the ``quantum group structure'' might not extend to it. There are canonical operator algebraic constructions that are ``compatible'' with the quantum group structure, however.

Let $\varpi$ be the universal $*$-representation (obtained, for example, by taking a direct sum of all GNS representations of states on $\mathcal{O}(G)$). We denote the corresponding $C^*$-algebra by $C_\varpi^*(G) = C_u(G)$. Then $\Delta_G$, $h_G$, and $\epsilon_G$ all extend continuously to $C_u(G)$. By construction, all unital $*$-representations on $\mathcal{O}(G)$ extend to $C_u(G)$.

Now let $\lambda_G$ be the GNS representation coming from $h_G$ and Denote $C_r(G) := C^*_{\lambda_G}(G)$. We denote GNS Hilbert space by $L^2(G)$. Both, $\Delta_G$ and $h_G$ extend continuously to $C_r(G)$. We also denote the von Neumann algebra $L^\infty(G) = C_r(G)''\subseteq B(L^2(G))$. It turns out that $\Delta_G$ and $h_G$ admit weak$^*$ continuous extensions to $L^\infty(G)$ as well. Unlike the universal construction, the counit does not always extend continuously  to $C_r(G)$. In fact, $\epsilon_G$ extends continuously to $C_r(G)$ if and only if $G$ is coamenable, which is equivalent to having $C_r(G) = C_u(G)$ (see \cite{BMT01}).

Another thing worth mentioning is the failure of the extendability of $S_G$ to $C_u(G)$ and $C_r(G)$. Indeed, $S_G$ extends continuously to $C_u(G)$ and $C_r(G)$ if and only if $S_G^2 = id$, which is equivalent to the Kac property of $G$.

The normal functionals $l^1(\hat G) = l^\infty(\hat G)_* = \oplus^{l^1}_{u\in Irr(G)}(M_{n_u})_*$ is an algebra with respect to convolution. Furthermore, we have $\mathcal{O}(G) \cong \oplus_{u\in Irr(G)}(M_{n_u})_*$ and hence $\mathcal{O}(G)$-$c_{00}(\hat G)$ duality given by
$$\langle x\triangleleft a, b\rangle = \langle x, ab\rangle, x\in c_{00}(\hat G), a,b\in \mathcal{O}(G).$$
If we identity $a\in \mathcal{O}(G)$ with $f_a\in l^1(\hat G)$, the duality bracket satisfies $\langle x, ab\rangle = f_a*f_b(x)$.

Let $G$ be a CQG and $\mathcal{O}(G)$ the associated Hopf $*$-algebra.
\begin{definition}
    A (right) {\bf coideal} is a unital $*$-subalgebra $A \subset \mathcal{O}(G)$ such that
    $$\Delta_G(A)\subseteq A\otimes \mathcal{O}(G).$$
    We can also consider a left version $A = S_G(A)$ (which satisfies $\Delta_G(S_G(A))\subseteq \mathcal{O}(G)\otimes A$).
\end{definition}
\begin{definition}
    Let $G$ and $H$ be CQGs. We write $H\leq G$ and say $H$ is a {\bf closed quantum subgroup} of $G$ if there exists a unital surjective $*$-homomorphism $q_H : \mathcal{O}(G)\to \mathcal{O}(H)$ such that
    $$(q_H\otimes q_H)\circ \Delta_G = \Delta_H\circ q_H.$$
    In this case, the quotient algebra
    $$\mathcal{O}(H\backslash G) := \{a\in \mathcal{O}(G) : (q_H\otimes id)\Delta_G(a) = 1\otimes a\}$$
    is a right coideal. Any coideal of the form $\mathcal{O}(H\backslash G)$ for some $H\leq G$ is said to be of {\bf quotient type}.
\end{definition}
For coideals there is generally no associated CQG quotient of $\mathcal{O}(G)$, however, there is an object denoted $\mathcal{O}(H)$ that does take the place of a closed quantum subgroup.

The following can be found in \cite{dCDT22, Ch18}, however, the ideas have been around for much longer (see \cite{T79}). Fix a coideal $A$ and set
$$\mathcal{O}(H) = \mathcal{O}(G) / \mathcal{O}(G)A_+$$
where $A_+ = A\cap \ker(\epsilon_G)$. It turns out that $\mathcal{O}(H)$ is a (left) $\mathcal{O}(G)$-module $*$-coalgebra. Let $\Delta_H$ be the coproduct on $\mathcal{O}(H)$ and
$$q_H : \mathcal{O}(G)\to \mathcal{O}(H)$$
the quotient map, which we note preserves the coproducts and is a $\mathcal{O}(G)$-module map. On the other hand, we can recover $A$ from $\mathcal{O}(H)$ by the ``quotient'' algebra
$$A = \mathcal{O}(H\backslash G) := \{a\in \mathcal{O}(G) : (q_H\otimes id)\Delta_G(a) = q_H(1)\otimes a\}.$$
It is clear that $\mathcal{O}(H\backslash G)$ is quotient type if and only if $q_H$ is a $*$-homomorphism. From now on, we will denote any coideal by $\mathcal{O}(H\backslash G)$.

We will occasionally use the following terminology to simplify exposition.
\begin{definition}
    Let $G$ be a CQG and $\mathcal{O}(H\backslash G)$ a coideal. We will call $\mathcal{O}(H)$ a {\bf quantum quotient} of $\mathcal{O}(G)$.
\end{definition}
\begin{remark}\label{Quotient Comodule Remark}
    The identification $\mathcal{O}(H\backslash G) \iff \mathcal{O}(H)$ is a bijection between the set of all coideals of $\mathcal{O}(G)$ and the set of all $*$-coalgebra $\mathcal{O}(G)$-module quotients of $\mathcal{O}(G)$. In particular, one could begin by formulating $*$-coalgebra $\mathcal{O}(G)$-module quotients of $\mathcal{O}(G)$ and then build the coideals as ``quotient'' algebras as above (see, for example, \cite{Ch18} for more on this).
\end{remark}
There is a functional $h_H : \mathcal{O}(H)\to \mathbb{C}$ that is left and right invariant:
$$(h_H\otimes id)\circ\Delta_H = h_H = (id\otimes h_H)\circ\Delta_H.$$
We will always use the normalization $h_H\circ q_H(1) = 1$. Then $\omega_H = h_H\circ q_H : \mathcal{O}(G)\to \mathbb{C}$ is a (not necessarily positive) idempotent functional such that
$$(h_H\otimes id)\circ\Delta_G : \mathcal{O}(G)\to \mathcal{O}(H\backslash G)$$
is a (not necessarily positive) projection.

As with any coalgebra, the linear dual space $\mathcal{O}(H)^*$ is an algebra via convolution:
$$\mu*\nu = (\mu\otimes \nu)\circ\Delta_H, ~ \mu,\nu\in \mathcal{O}(H).$$
Moreover, $\mathcal{O}(H)^*\subseteq \mathcal{O}(G)^*$ is a subalgebra.

We will say an irreducible $*$-representation $\alpha$ of $\mathcal{O}(H)^*$ is admissible if it is a subrepresentation of some $u|_{\mathcal{O}(H)^*}$, $u\in Rep(G)$. We will let $Irr(H)$ denote the admissible irreducible $*$-representations of $\mathcal{O}(H)^*$. There is an algebra isomorphism
$$\mathcal{F} : \{h_Ha : a\in \mathcal{O}(G)\} \to c_{00}(\hat H) := \bigoplus_{\alpha\in Irr(H)} M_{n_\alpha} \subseteq l^\infty(\hat G)$$
where $h_Ha \in \mathcal{O}(H)^*$ is the functional such that
$$(h_Ha)(x) = h_H(ax), ~ x\in \mathcal{O}(H)$$
(cf. \cite{dCDT22}). The $\mathcal{O}(G)$-module structure on $\mathcal{O}(H)$ makes $l^\infty(\hat H) = c_{00}(\hat H)''$ a coideal of $l^\infty(\hat G)$:
$$\Delta_{\hat G}(l^\infty(\hat H))\subseteq l^\infty(\hat G)\overline{\otimes} l^\infty(\hat H).$$
The coideal $l^\infty(\hat H)$ is called the {\bf codual} of $\mathcal{O}(H\backslash G)$. In terms of duality, this is expressed by
$$\langle x\triangleleft a, b\rangle = \langle x, ab\rangle, a\in \mathcal{O}(G), x\in c_{00}(\hat H), b\in \mathcal{O}(H)$$
where $x \triangleleft a = (\hat{\mathcal{F}}(a)\otimes id)\Delta_{\hat G}(x)$ and we are using the isomorphism
$$\hat{\mathcal{F}} : \mathcal{O}(G)\to \oplus_{u\in Irr(G)}(M_{n_u})_{*}\subseteq l^1(\hat G).$$
Finally, we also have
$$\mathcal{F}(h_Ha) = \mathcal{F}(h_H)\triangleleft a.$$
\end{subsection}

\begin{subsection} {Categorical duality}

In this paper, we deal with categories and module categories related to coactions of compact quantum groups
on unital $C^*$-algebras.

\begin{definition} \label{cov}
Let $A$ be a unital $C^*$-algebra and $(C_r(G),\Delta)$ be a CQG, and let $\delta:A\to A\otimes C_r(G)$ be a unital faithful $*$-homomorphism called a \it{coaction} of $G$ on $A$. The triple $(A,G,\delta)$ is called a $G$-$C^*$-algebra or a quantum homogeneous space if the following statements hold:
\begin{itemize}
    \item The map $\delta$ satisfies $(\delta\otimes id)\circ\delta= (id\otimes\Delta)\circ\delta$.
      \item The vector space spanned by $\delta(A)(\mathbb C\otimes C_r(G))$ is dense in $A\otimes C_r(G)$.
 \end{itemize}
\end{definition}

The following categorical duality theorem for $G$-$C^*$-algebras holds (see \cite{dCY13}, Theorem 6.4 and \cite[Theorem 3.3]{N14}):

\begin{theorem} \label{catdual}
Let $G$ be a reduced CQG. Then the following categories are equivalent:

(i) The category of unital $G$-$C^*$-algebras with unital $G$-equivariant $*$-homomorphisms as morphisms.

(ii) The category of pairs $(\mathcal M, Q)$, where $\mathcal M$ is a $Rep(G)$-module $C^*$-category and $Q$ is a
generator in $\mathcal M$, with equivalence classes of unitary $Rep(G)$-module functors respecting the prescribed generators
as morphisms.
\end{theorem}
In particular, given a unital $G$-$C^*$-algebra $(A,G,\delta)$, the corresponding $Rep(G)$-module $C^*$-category
$\mathcal M$ consists of finitely generated projective right Hilbert $A$-modules $(E,\delta_E)$ with ajointable equivariant
maps as morphisms. The module category structure, if $U\in Rep(G)$ and $E\in \mathcal M$ is given by $H_U\otimes E$.
$(A,\delta)$ is a generator in $\mathcal M$, i.e., any object of $\mathcal M$ is of the form $p(H_U\otimes A)$,
where $p$ is a projection. Its Grothendieck group denoted by $K^G_0 (A)$ is called an equivariant $K_0$-group.

There is the Green-Julg isomorphism $K^G_0 (A)\cong K_0(A\rtimes_\delta G)$, where the last one is the ordinary $K_0$-group
of the crossed product algebra $A\rtimes_\delta G$ (see, for instance, \cite[Th\'eor\`eme 5.10]{V04}).

\begin{remark}\label{cdualcoideal}
Any coideal $A$ of a CQG $G$ is related to the obvious injective unital $G$-equivariant $*$-homomorphism $A\mapsto C(G)$ of $G$-$C^*$-algebras. Then the above categorical duality implies that the pair $(\mathcal M, Q)$ associated with $A$ comes
together with a surjective unitary $Rep(G)$-module functor from $Rep(G)$ viewed as a $Rep(G)$-module category to $\mathcal M$
sending the trivial representation $\varepsilon$ of $G$ to $Q$.
\end{remark}

\end{subsection}
\end{section}
\begin{section} {Basic definitions and examples.}

\begin{subsection}{Basic definitions.}

\begin{definition} \label{fusalg} (see \cite{HI98} or \cite{K08})

A fusion algebra is a unital algebra $R$ with a basis $I$ over $\mathbb Z$ such that:
$$
1)\quad \zeta\eta=\Sigma_{\alpha}\ N^{\alpha}_{\zeta,\eta}\alpha\quad\quad \forall \zeta,\eta\in I,
$$
where $N^{\alpha}_{\zeta,\eta}\in \mathbb Z^+$, only finitely many nonzero.

2) There is a bijection $\zeta\mapsto\overline\zeta$ of $I$ which extends to a $\mathbb Z$-linear
anti-multiplicative involution of $R$.

3)  Frobenius reciprocity:
$$
N^{\alpha}_{\zeta,\eta}=N^{\eta}_{\overline\zeta,\alpha}=N^{\zeta}_{\alpha,\overline\eta}\quad\quad \forall \zeta,\eta,\alpha\in I. $$
4) There is a {\it dimension function} $d:I\to[1,\infty[$ such that $d(\zeta)=
d(\overline\zeta)$ which extends to a $\mathbb Z$-linear multiplicative map $R\to\mathbb R$.
\end{definition}

Let now $(R,I^R, N^w_{u,v}, d^R)$ be a fusion algebra with a dimension function
\begin{definition} \label{module}
\vskip 0.5cm
a) A left fusion module over $R$ is a collection $(M, I^M, c^\beta_{u,\alpha})$, where:
\vskip 0.5cm
(i) $M$ is a free left $\mathbb Z$-module over $R$, $I^M$ is a basis in $M$.
 \vskip 0.5cm
(ii) $c^\beta_{u,\alpha} (u\in I^R,\alpha,\beta\in I^M)$ are non-negative integers which define a module structure
on $M$, i.e., $u\circ\alpha=\underset {\beta\in I^M}\sum c^\beta_{u,\alpha} \beta$ (this sum is finite).
\vskip 0.5cm
(iii) The following relation holds: $c^\beta_{u,\alpha}=c^\alpha_{\overline u,\beta}$ for all $u\in I^R,\alpha,
\beta\in I^M$, where $u\mapsto \overline u$ is the involution in $I^R$.
\vskip 0.5cm
b) A dimension function on $M$ is a linear form $d^M:M\to\mathbb C$ such that $d^M(\alpha)>0$ and $d^M(u\circ\alpha)=d^R(u)d^M(\alpha)$.
\end{definition}
Note that the conditions (ii),(iii) and b), and the relation $d^R(u)=d^R(\overline u)$ imply that
$$
\underset {\beta\in I^M}\sum c^\beta_{u,\alpha}\frac{d^M(\beta)}{d^R(u)d^M(\alpha)}=
\underset {\beta\in I^M}\sum c^\alpha_{\overline u,\beta}\frac{d^M(\beta)}{d^R(u)d^M(\alpha)}=1.
$$

Define the left action operator
$$
\Gamma_u: I_R\to M,\quad u\mapsto u\circ \alpha
$$
which is densely defined on $l^2(I^M)$ and let $\Gamma_U$ be its extension to $R$ by linearity. Using \cite[Lemma 2.7.3]{NT13}, it is easy to prove the following analog of \cite[Proposition 2.7.4]{NT13}:
\begin{proposition} \label{gamma}
The operator $\Gamma_u$ extends to a bounded linear operator on $l^2(I^M)$ and for the above dimension function $d^R$ we have
$$
||\Gamma_u||_{\mathcal B(l^2(I^M))}\leq d^R(u)\quad\text{for\ all}\quad u\in I^R.
$$
\end{proposition}
\begin{proof}
    Indeed, it suffices to apply \cite[Lemma 2.7.3]{NT13}, to the matrix $\frac{1}{d^R(u)}\Gamma_u= [\frac{c^\beta_{u,\alpha}}{d^R(u)}]$ and to check, using condition (iii) above, that for the finite
vectors $f=g=(d^M(\alpha)\ (\alpha\in I^M)$ we have $\frac{1}{d^R(u)}\Gamma_u(f)=g$ and $\frac{1}{d^R(u)}
(\Gamma_u)^t(g)=f$.

Note that the condition (iii) above means that $(\Gamma_u)^*=\Gamma_{\overline u}$.
\end{proof}
Fix now a dimension function $d^M$ on $M$. For a probability measure $\mu$ on $I^R$
define a contraction $\Gamma_\mu$ on $l^2(I^M)$ by
$$
\Gamma_\mu=\underset{u\in I^R}\sum\frac{\mu(u)}{d^R(u)}\Gamma_u.
$$
\begin{definition} \label{amen}
\vskip 0.5cm
The pair $((M, I^M, c^\beta_{u,\alpha}), d^M)$ is called amenable fusion module over the fusion
algebra $(R,I^R, N^w_{u,v}, d^R)$ if the following condition holds:
$$
||\Gamma_\mu|| = 1\quad\text{for\ every\ probability\ measure}\ \mu\ on\ I^R.
$$
\end{definition}
Let us reformulate the definition of an amenable fusion module in terms of multiplication by fusion elements instead of ``convolution'' by probability measures. In the classical picture of a discrete dual $\hat G$, this is akin to the difference
between considering the extension of the left regular representation to $\mathbb{C}[G]$ versus to $\mathrm{Prob}(G)$.

Define the following operators for arbitrary $u = \sum_{v\in I^R} c_v v \in R$:
$$
\Gamma_u = \sum_{v\in I^R} c_v \Gamma_v.
$$
Note that by scaling, the definition of an amenable fusion ring in Definition~\ref{amen} is clearly equivalent to having $||\Gamma_\mu||_{\mathcal{B}(l^2(I^R))} = ||\mu||_1$ for every bounded positive measure $\mu$ on $I^R$.

Recall that the finitely supported probability measures on $I^R$ are norm dense in $\mathrm{Prob}(I^R)$. Moreover, an easy calculation shows that $\mu\mapsto \Gamma_\mu$ is norm-norm continuous. In what follows, let $R^+ = \{\sum_{v\in I^r} c_v v : c_v\geq 0\}$.
\begin{proposition} \label{reform}
    We have that $((\mathrm{Rep}(M), I^M, c^\beta_{u,\alpha}), d^M)$ is amenable if and only if $||\Gamma_u|| = d^R(u)$ for every $u\in R^+$.
\end{proposition}
\begin{proof}
    Fix $u = \sum_{v\in I^R}c_v v\in R$ and define the positive measure $\mu_u(v) = c_v d^R(v)$ for all $v\in I^R$, which has norm
    \begin{align*}
        ||\mu_u||_1 &= \sum_{v\in I^R}c_v d^R(v) = d^R(\sum_{v\in I^R} c_v v) = d^R(u).
    \end{align*}
    Note also that
    $$\Gamma_{\mu_u} = \sum_{v\in I^R} \frac{\mu_u(v)}{d^R(v)}\Gamma_v = \sum_{v\in I^R} c_v \Gamma_v = \Gamma_u.$$
    Therefore, if $((\mathrm{Rep}(M), I^M, c^\beta_{u,\alpha}), d^M)$ is amenable then
    \begin{align*}
        ||\Gamma_u||_{\mathcal{B}(l^2(I^M))} &= d^R(u).
    \end{align*}
    Conversely, let $\mu$ be a finitely supported positive measure on $I^R$. Similarly to above, we can calculate $||\mu||_1 = d^R(v)$ where $v = \sum_{u\in I^R}\frac{\mu(u)}{d^R(u)}u\in R^+$. Then,
    \begin{align*}
        ||\Gamma_\mu||_{\mathcal{B}(l^2(I^M))} &= ||\sum_{u\in I^R} \frac{\mu(u)}{d^R(u)}\Gamma_u||_{\mathcal{B}(l^2(I^M))}
        \\
        &= ||\Gamma_v||_{\mathcal{B}(l^2(I^M))}
        \\
        &= d^R(v)
        \\
        &= ||\mu||_1.
    \end{align*}
    In particular, we have $||\Gamma_\mu||_{\mathcal{B}(l^2(I^M))} = 1$ for every finitely supported probability measure $\mu$ on $I^R$. A simple density and continuity argument shows that $||\Gamma_\mu||_{\mathcal{B}(l^2(I^M))} = 1$ for every $\mu\in \mathrm{Prob}(I^R)$.
\end{proof}

There are natural definitions of a direct sum, of an isomorphism and of irreducibility of fusion modules.
\begin{remark} \label{countable}
In this paper, we suppose that the sets $I^A$ and $I^M$ are countable which is the case for representation categories of separable quantum groups.
\end{remark}
\end{subsection}
\begin{subsection} {Examples.}





{\bf 1.} Let $\mathcal C$ be a rigid $C^*$-tensor category and $R=R(\mathcal C)$ its fusion ring (see \cite[Definition 2.7.2]{NT13}) with basis $I^R=Irr(\mathcal C)$ and structural constants $N^w_{u,v}$. The set of dimension functions on $I^R$ is not empty, it contains
at least the {\it intrinsic or quantum} dimension function (see \cite{NT13}).
But in this work we mainly deal with $\mathcal C=Rep(G)$, where $G$ is a CQG, and we use the dimension function $d^R(u)=dim H_u$, where $u\in Irr(G)$.

Let $\mathcal M$ be a semi-simple module $C^*$-category over $\mathcal C$ (see \cite[Definition 2.14]{dCY13}). The Grothendieck group of $\mathcal M$ is a $\mathbb Z_+$-module over $R(\mathcal C)$ viewed as a $\mathbb Z_+$-ring (see \cite[Definition 3.4.2]{EGNO15}) with basis $I^M=Irr(\mathcal M)$ and structural constants $c^\beta_{u,\alpha}$. We denote its complexification by $M$.
\begin{definition} \label{mod}
Module category $\mathcal M$ is said to be {\it amenable} if $(M, I^M, c^\beta_{u,\alpha})$ is equipped with amenable dimension function in the sense of Definition~\ref{amen}.
\end{definition}

{\bf 2.} In particular, let $(A,G,\delta)$ be a $G$-$C^*$-algebra, $\mathcal C=Rep(G)$, where $G$ is a CQG. Let $\mathcal M$
be corresponding the left module category constructed by the categorical duality and $K^G_0 (A)$ its Grothendieck group.
The left module structure of $K^G_0 (A)$ over $R(G)=\mathbb C [Rep(G)]$ is given by $[U]\cdot[E]=[H_U\otimes E]$, and the Green-Julg isomorphism is an $R(G)$-module map - see \cite[Lemma 4.2]{T08}. There exists an index set $I$ and Hilbert spaces $H_i$ for each
$i\in I$ such that $A\rtimes_\delta G$ is isomorphic to $\underset{i\in I}\oplus K(H_i)$ \cite[Theorem 19]{B95}. Let us fix minimal projections $e_i\in K(H_i)$ for all $i\in I$. Hence we have
$K^G_0 (A)\cong\underset{i\in I}\oplus\mathbb Z[ei]$ by the Green-Julg theorem. One can check that in $K^G_0 (A)$ $[e_i]$ becomes a $G$-equivariant $A$-module $[e_i(A\otimes L_2(G))]$.

The $R(G)$-module structure of $K^G_0 (A)$ gives the equality
$$
U\cdot[e_i] =\underset{j\in I}\sum M(U)_{i,j}[e_j],
$$
where the matrix $M(U) = (M(U)_{i,j}) (i,j\in I)$ is not necessarily of finite size, but for any fixed $i$ the sum is finite. \cite[Corollary 4.4]{T08} claims that if $(A,G,\delta)$ is a compact quantum ergodic system and $M : R(G) \to M|I|(\mathbb Z)$ is the multiplicity map, then for all $i, j \in I$ we have $M(U)_{i,j} = dim Hom(H_U, e_i(A\otimes K(L_2(G)))e_j)$, which implies that the multiplicity map is a $*$-homomorphism. \cite[Lemma 4.5]{T08} claims that the $G$-space $e_i(A\otimes K(L_2(G)))e_j)$ is not $0$ for all $i, j \in I$.

\cite[Corollary 4.21]{T08} in fact shows how to construct a dimension function on the fusion module $K^G_0(A)\cong K_0(A\rtimes_\delta G)$, when $A = \mathcal{O}(H\backslash G) \subseteq C(G)$ is a right coideal. Let $\Lambda$ be the inclusion matrix of cross products $A\rtimes_\delta G\cong\underset{i\in I^A}\oplus K(H_i) \subset C(G)\rtimes_\delta G\cong\underset{j\in J}\oplus K(H_j)$, that is, $K(H_i)$ is amplified into $K(H_j)$ by $\Lambda_{j,ix}$ times. It is actually a row vector because $C(G)\rtimes_\delta G\cong K(L_2(G))$. Its components $c_i\ (i\in I:=I^A)$
are strictly positive integers such that there is $i_0\in I^A$ for which $c_{i_0}=1$ and $M(U)c=d_U c$, where $U\in I^R$. This implies the equality $c(U\cdot[e_i])=d_U c_i$.


In particular, let $\mathcal{O}(H\backslash G) $ be a quotient type coideal corresponding to a closed quantum subgroup $H$ of $G$.
Then the above mentioned fusion module over $R(G)$ will be $(\mathbb{C}[\mathrm{Rep}(H)], I=\mathrm{Irr}(H),
N^\beta_{u|_H,\alpha})$ with the dimension function $d^H(\alpha) = \dim(\alpha)$, where $u\in Irr(G),\alpha,
\beta\in Irr(H)$ (see \cite[Proposition 4.15]{T08} and just after it).

Now we will show in which way this result can be extended to arbitrary coideals. In what follows we use
\cite{V05} as a reference.

Let $\mathcal{O}(H\backslash G) \subset C(G)$ be a right coideal. We will set $c_{0}(\hat H) = \overline{c_{00}(\hat H)} \subseteq l^\infty(\hat G)$. First, we will build a surjective unitary module functor $Rep(G) \to Rep(H)$. Recall the discussion in the previous subsection. Let $u\in Rep(G)$ and take the corresponding $*$-representation $\pi_u : l^\infty(\hat G) \to B(H_u)$. The restriction $\pi_u|_{l^\infty(\hat H)}$ to $l^\infty(\hat H)$ is a finite dimensional $*$-representation, and we let $u|_H\in Rep(H)$ be the corresponding element. We realize $Rep(H)$ as a $Rep(G)$ tensor module by setting
$$\pi_{u\circ \alpha} = (\pi_u\otimes \pi_\alpha)\circ\Delta_{\hat H}, ~ u\in Rep(G), \alpha\in Rep(H).$$
Then, for $u,v\in Rep(G)$,
$$\pi_{u\cdot v}|_{l^\infty(\hat H)} = (\pi_u\otimes \pi_v)\circ\Delta_{\hat G}|_{l^\infty(\hat H)} = (\pi_u\otimes \pi_{v|_H})\circ\Delta_{\hat H},$$
which shows $(u\cdot v)|_H = u\circ v|_H$. Finally, we note that the restriction of the trivial representation on $G$ generates $Rep(H)$ as a $C^*$-module category.

Altogether, this shows
$$Rep(G)\ni u\mapsto u|_H\in Rep(H)$$
is a unitary module functor.

Consider the crossed product $C^*$-algebra
$$\mathcal{O}(H\backslash G) \rtimes_\Delta G=
\overline{\{\Delta(\mathcal{O}(H\backslash G) )\cup (1\otimes\hat J c_0(\hat G) \hat J)\}}.$$
Recall that $\Delta(\mathcal{O}(H\backslash G) )=W^*(1\otimes \mathcal{O}(H\backslash G) )W$,
where $W$ is the left fundamental multiplicative unitary of G. Then, as $\hat J c_0(\hat G)\hat J\subset [c_0(\hat G)]'$,
we have $W(\mathcal{O}(H\backslash G) \rtimes_\Delta G)W^*=\overline{\{W(\Delta(\mathcal{O}(H\backslash G) ))W^*\cup W(1\otimes\hat J c_0(\hat G) \hat J)W^* \}}$ = $1\otimes
\overline{\{\mathcal{O}(H\backslash G) \cup\hat J c_0(\hat G)\hat J\}}$. Here $\hat J : L^2(G)\to L^2(G)$ is the modular conjugated associated to $h_G$.

Notice that $R_{\hat G}(l^\infty(\hat H))$ is a right coideal of $l^\infty(\hat G)$ and satisfies
$$L^\infty(H\backslash G)' \cap l^\infty(\hat G) = R_{\hat G}(l^\infty(\hat H)) ~ \cite{ILP98}$$
where $R_{\hat G}$ is the unitary antipode of $\hat G$. Note that in \cite{ILP98} they prove the analogue of the above equality using a different convention for duality. Then
$$\{ C(H\backslash G)\cup c_0(\hat G)\}' = R_{\hat G}(l^\infty(\hat H)).$$

Note that an analogue of the above equality was observed in \cite[Remark 3.8]{KS12} in the context of locally compact quantum groups using instead the right regular representation.

Since $\{C(H\backslash G)\cup\ c_0(\hat G)\}'= R_{\hat G}(l^\infty(\hat H))$, von Neumann algebras generated by $C(H\backslash G)\rtimes_\Delta G$
and $c_0(\hat H)$ are Morita-equivalent. Explicitly, define
$$
\mathcal T=\{v\in B(l^2(\hat H),L^2(G))|vx=\hat q'_H(x)v\ \text{for\ all}\ x\in R_{\hat G}(c_0(\hat H))\}.
$$
Defining $<u,v>:=u^*v$, we have $u^*vx=xu^*v$ for all $x\in R_{\hat G}(c_0(\hat H))$, so the space $\mathcal T$ becomes a $W^*$-$R_{\hat G}(l^\infty(\hat H))$-module. As $\{C(H\backslash G)\cup\hat c_0(\hat G)\}''=[R_{\hat G}(c_0(\hat H))]'$, we have an imprimitivity
$W^*-A\rtimes_\Delta G - R_{\hat G}(c_0(\hat H))$-bimodule. This implies that their categories of representations are equivalent so that $Rep(H)$ can be considered as a $C^*$-module category over $Rep(G)$ equipped with the restriction functor as a unitary
module surjective functor $Rep(G)\to Rep(H)$.



In particular, this shows that the fusion $Rep(G)$-module $C^*$-category associated to a coideal $\mathcal{O}(H\backslash G)$ is
$$
(\mathbb{C}[Rep(H)], Irr(H), c^\beta_{u,\alpha}).
$$
Then the dimension function on $\mathbb{C}[Rep(H)]$ that we consider is the linear form
$$d_H(\alpha) = dim(H_\alpha) = n_\alpha, ~\alpha\in Rep(H).$$
\end{subsection}
\end{section}

\begin{section} {Coideals of compact and discrete quantum groups}
\begin{subsection}{$g$-Quasi-Regular Representations}
The main topic of interest for us for will be a notion of coamenability of a $*$-coalgebra $H$ as defined in the previous section. To formulate this property, we are forced to restrict our attention to the coideals that come with a so-called relatively invariant integral.

Before proceeding, a {\it character} $g$ of $G$ is an element of $\prod_{u\in Irr(G)}M_{n_u}$ such that $\Delta_{\hat G}(g) = g\otimes g$ and $\langle g, 1\rangle = 1$. A character $g$ identifies with a unital homomorphism $\varphi_g : \mathcal{O}(G)\to \mathbb{C}$, where
$$\varphi_g(a) = \langle g, a \rangle,~ a\in \mathcal{O}(G).$$
We let
$$\mathcal{F}(h_H) = P_H$$
where
$$\mathcal{F} : \{h_Ha : a\in \mathcal{O}(G)\} \to c_{00}(\hat H)$$
is the isomorphism mentioned in Section $2.2$. It turns out $P_H$ is a {\bf group-like projection}, i.e., $P_H^* = P_H^2 = P_H$ and
\begin{align}\label{Grouplike Projections}
    (1\otimes P_H)\Delta_{\hat G}(P_H) = P_H\otimes P_H.
\end{align}
A lot of work has been done on group-like projections, including on those in the setting of locally compact quantum groups (for example, see \cite{dCDT22, FK17}).
\begin{definition} (\cite[Definition 1.8]{dCDT22}).
    Let $G$ be a CQG and $\mathcal{O}(H\backslash G)$ a coideal. Let $g \in $ be a character. We say a functional $\Psi : c_{00}(\hat H)\to \mathbb{C}$ is $g$-invariant if
    $$\Psi(x\triangleleft a) = \varphi_g(a)\Psi(x), ~ x\in c_{00}(\hat H), a\in \mathcal{O}(G).$$
    A functional $\psi$ is {\bf relatively invariant} if it is $g$-invariant for some character $g$ and {\bf invariant} if it is $1$-invariant. In all instances where $g = 1$, we simply drop the $g$-notation.
    
    If $\Psi\geq 0$ and is relatively invariant, then we say $\Psi$ is a {\bf relatively invariant integral}.

    We will say $l^\infty(\hat H)$ is {\bf $g$-integrable} if $c_{00}(\hat H)$ admits a $g$-invariant integral. We will say $l^\infty(\hat H)$ is {\bf relatively integrable} if $c_{00}(\hat H)$ admits a $g$-integrable integral for some character $g$.
\end{definition}
\begin{theorem}(\cite[Theorem 1.11]{dCDT22})
    Let $G$ be a CQG and $\mathcal{O}(H\backslash G)$ a coideal. The following hold:
    \begin{itemize}
        \item a $g$-invariant functional $\Psi : c_{00}(\hat H)\to \mathbb{C}$ is positive if and only if $g$ is positive and invertible as an element of $\prod_{u\in Irr(G)}M_{n_u}$;
        \item a relatively invariant integral $c_{00}(\hat H)\to \mathbb{C}$ exists if and only if there exists a positive and invertible character $g$ such that $S_{\hat G}(P_H) = g^{-1}P_H$;
        \item when a $g$-invariant integral exists, it is unique up to scaling. A canonical choice is:
        $$\Psi_{\hat H}(P_H \triangleleft a) = \varphi_g(a), ~ a\in \mathcal{O}(G).$$
    \end{itemize}
\end{theorem}
\begin{remark}
    The functional $\psi_{\hat H}$ is moreover faithful, as can be seen from the concrete description found in the proof of \cite[Theorem 1.11]{dCDT22}. We will give the concrete description later once it is of use.
\end{remark}
Now we will construct a $g$-quasi-regular representation. Fix a coideal $\mathcal{O}(H\backslash G)$ and suppose there exists a $g$-invariant integral $\psi_{\hat H}$. Let
$$\eta_{\hat H}^g : c_{00}(\hat H)\to l^2_g(\hat H)$$
be the GNS construction coming from $\psi_{\hat H}$.

In what follows, recall that
$$g\hat{\triangleright} a = (id\otimes \varphi_g)\Delta_G(a) ~ \text{and} ~ a\hat{\triangleleft} g = (\varphi_g\otimes id)\Delta_G(a).$$
Also, we will write $g\circ S_G$ to denote the character in $\prod_{\alpha\in Irr(H)}M_{n_\alpha}$ corresponding to $\varphi_g\circ S_G$. Furthermore, since $S_{\hat G}(g^{-1}) = g$,
$$(a\hat{\triangleleft} g^{-1})^* = a^*\hat{\triangleleft} S_{\hat G}(g^{-1})^* = a^*\hat{\triangleleft} g$$
and by symmetry, $(g^{-1}\hat{\triangleright} a)^* = g\hat{\triangleright} a^*$. Also, because $\varphi_g$ is a character, $g\circ S_G = g^{-1} = g\circ S_G^{-1}$.
\begin{proposition}\label{gQuasiRegular Representation Proposition}
    The map
    $$\lambda_{H}^g : \mathcal{O}(G)\to B(l^2_g(\hat H)), ~ \lambda^g_H(a)\eta_{\hat H}^g(x) = \eta_{\hat H}^g(x\triangleleft (g^{-1/2}\hat{\triangleright}S_G^{-1}(a)))$$
    is a unital $*$-representation.
\end{proposition}
\begin{proof}
    Define the anti-homomorphism
    $$\theta : \mathcal{O}(G)\to B(l^2_g(\hat H)), ~ \theta(a)\eta_{\hat H}^g(x) = \eta_{\hat H}^g(x\triangleleft a ), ~a\in \mathcal{O}(G), x\in c_{00}(\hat H).$$
    Let us compute $\theta(a)^*$:
\begin{align*}
    \langle \theta(a)\eta_{\hat H}^g(x), \eta_{\hat H}^g(y)\rangle &= \psi_{\hat H}(y^*(x\triangleleft a))
    \\
    &= \psi_{\hat H}( (y^* \triangleleft S_G(a_{(1)}) a_{(2)})( x\triangleleft a_{(3)} ))
    \\
    &= \psi_{\hat H}( (y^* \triangleleft  S_G(a_{(1)})(a_{(2)} )_{(1)}) (x \triangleleft (a_{(2)})_{(2)}))
    \\
    &= \psi_{\hat H} (((y^*\triangleleft  S_G(a_{(1)})) x)\triangleleft (a_{(2)} ) )
    \\
    &= \varphi_g( a_{(2)} ) \psi_{\hat H}( (y^*\triangleleft S_G(a_{(1)})) x )
    \\
    &= \varphi_g( a_{(2)} ) \psi_{\hat H}( (y\triangleleft S_G^2(a_{(1)})^*)^* x )
    \\
    &= \langle \eta_{\hat H}^g(x), \varphi_g^{-1}(a_{(2)}^*)\eta_{\hat H}^g(y\triangleleft S_G^{-2}(a_{(1)}^*) )\rangle.
\end{align*}
    Then, by the above properties of $\varphi_g$,
    $$\theta(a)^*  = \theta(g^{-1}\hat{\triangleright} S_G^{-2}(a^*)) = \theta( S_G^{-2}( g^{-1}\hat{\triangleright} a^*) ) )$$
    and hence
    \begin{align*}
        \lambda_H^g(a) = \theta(g^{-1/2}\hat{\triangleright}S_G^{-1}(a))^* &= \theta(S_G^{-2}(g^{-1}\hat{\triangleright}(g^{-1/2}\hat{\triangleright} S_G^{-1}(a))^*))
        \\
        &= \theta(S_G^{-2}(g^{-1}\hat{\triangleright}(g^{1/2}\hat{\triangleright}S_G(a^*))))
        \\
        &= \theta(g^{-1/2}\hat{\triangleright}S_G^{-1}(a^*)).
    \end{align*}
 Clearly, $\lambda_H^g$ is a unital homomorphism and so this concludes the proof.
\end{proof}
\begin{definition}\label{g-Quasi-Regular Representation}
    Let $\mathcal{O}(H\backslash G)$ be a coideal such that $l^\infty(\hat H)$ is $g$-integrable. We call
    $$\lambda_{H}^g : \mathcal{O}(G)\to B(l^2_g(\hat H)), ~ \lambda^g_H(a)\eta_{\hat H}^g(x) = \eta_{\hat H}^g(x\triangleleft (g^{-1/2}\hat{\triangleright}S_G^{-1}(a)))$$
    the $g$-quasi-regular representation induced by $H$.
    
    If $g = 1$, then we call $\lambda^g_H = \lambda_H$ the quasi-regular representation induced by $H$.
\end{definition}
Define the functional
$$\omega^g_H(a) :=  \omega_H(g^{-1/2} \hat{\triangleright} S_G^{-1}(a)), ~ a\in \mathcal{O}(G).$$
Next, we will establish a Plancherel's theorem for relatively integrable coideals. This result was established previously for coideals where $\omega_H$ is a state (cf. \cite{Ch18}).
\begin{proposition}\label{Plancherel Theorem}
    Let $\mathcal{O}(H\backslash G)$ be a coideal such that $l^\infty(\hat H)$ is $g$-integrable. Then
    $$\psi_{\hat H}((P_H\triangleleft (g^{-1/2}\hat{\triangleright} S_G^{-1}(a)))^*(P_H\triangleleft (g^{-1/2}\hat{\triangleright} S_G^{-1}(b)))) = \omega_H(g^{-1/2}\hat{\triangleright}(a^*b)).$$
    In particular, there is a unitary isomorphism
    $$l^2_g(\hat H)\to L^2(\mathcal{O}(G), \omega_H^g), ~ \eta_{\hat H}^g(P_H\triangleleft a)\mapsto \eta_H^g(q_H(g\hat{\triangleright} S_G^{-1}(a)))$$
    where $L^2(\mathcal{O}(G), \omega_H^g)$ is the GNS Hilbert space of $\omega^g_H$.
\end{proposition}
\begin{proof}
    Since $\lambda_H^g$ is involutive and $P_H^* = P_H$,
    \begin{align*}
        &\psi_{\hat H}( (P_H\triangleleft (g^{-1/2}\hat{\triangleright} S_G^{-1}(a)))^* (P_H\triangleleft (g^{-1/2}\hat{\triangleright} S_G^{-1}(b))))
        \\
        &= \psi_{\hat H}( P_H (P_H\triangleleft ((g^{-1/2}\hat{\triangleright} S_G^{-1}(b))(g^{-1/2}\hat{\triangleright} S_G^{-1}(a^*)))))
        \\
        &= \psi_{\hat H}( P_H (P_H\triangleleft (g^{-1/2}\hat{\triangleright} S_G^{-1}(a^*b)) ) )) 
        \\
        &= \omega_H(g^{-1/2}\hat{\triangleright}(S_G^{-1}(a^*b) )) ~ \text{(by \eqref{Grouplike Projections}.)}
        \\
        &= \omega_H^g(a^*b).
    \end{align*}
\end{proof}
An immediate deduction from Proposition~\ref{Plancherel Theorem} is that $$\omega^g_H(a) =\langle \lambda^g_H(a)\eta_{\hat H}(P_H), \eta_{\hat H}(P_H)\rangle,$$
hence $\omega^g_H$ is a state that has $\lambda^g_H$ as its GNS representation.

Consider the unitary corepresentation operator corresponding to $\lambda^g_H$:
$$W_H^g := (id\otimes \lambda^g_H)(\mathbb{W}_{\hat G})$$
where $\mathbb{W}_{\hat G}$ is the universal fundamental multiplicative unitary for $\hat G$. Consider the induced coaction
$$ad_{W_H^g} : B(l^2_g(\hat H)) \to l^\infty(\hat G)\overline{\otimes} B(l^2_g(\hat H)), ~ T\mapsto (W_H^g)^*(1\otimes T)W_H^g.$$
Let $\sigma_{\hat H}^g : l^\infty(\hat H)\to B(l^2_g(\hat H))$ denote the GNS representation induced by $\psi_{\hat H}$. The coaction $ad_{W^g_H}$ descends to a coaction on $l^\infty(\hat H)$.

Take $a\in \mathcal{O}(G)$, $z\in c_{00}(\hat G)$, and $x\in c_{00}(\hat H)$. Since $\langle x, g^{-1/2}\hat{\triangleright}a\rangle = \langle xg^{-1/2}, a\rangle$, it is straightforward to calculate that
$$(W^g_H)^*(\eta_{\hat G}(z)\otimes \eta_{\hat H}(x)) = (\eta_{\hat G}\otimes \eta_{\hat H})(\Delta_{\hat G}(x)(g^{-1/2}z\otimes 1)).$$
With this formula, we find that $ad_{W^g_H}$ induces the same coaction on $l^\infty(\hat H)$ as the one defined by the restriction of $\Delta_{\hat G}$ to $l^\infty(\hat H)$.
\begin{proposition}\label{Coaction on Hhat}
    We have that $(id\otimes \sigma^g_{\hat H})\circ\Delta_{\hat G}|_{l^\infty(\hat H)} = ad_{W^g_H}\circ \sigma^g_{\hat H}$.
\end{proposition}
\begin{proof} 
    Take $a\in \mathcal{O}(G)$, $z\in c_{00}(\hat G)$, and $x,y\in c_{00}(\hat H)$. Then,
    \begin{align*}
        &(W^H_g)^*(1\otimes \sigma^g_{\hat H}(x) )(\eta_{\hat G}(z)\otimes \eta_{\hat H}(y))
        \\
        &= (\eta_{\hat G}\otimes \eta_{\hat H})(\Delta_{\hat G}(xy) (g^{-1/2}z\otimes 1))
        \\
        &= (id\otimes \sigma^g_{\hat H})(\Delta_{\hat G}(x))(\eta_{\hat G}\otimes \eta_{\hat H})(\Delta_{\hat G}(y) (g^{-1/2}z\otimes 1))
        \\
        &= (id\otimes \sigma^g_{\hat H})(\Delta_{\hat G}(x))(W^g_H)^*(\eta_{\hat G}(z)\otimes \eta_{\hat H}(y)).
    \end{align*}
    Hence $(id\otimes \sigma^g_{\hat H})(\Delta_{\hat G}(x)) = ad_{W^g_H}(\sigma^g_{\hat H}(x))$.
\end{proof}
If $g = 1$, $l^\infty(\hat H)$ is said to be integrable and the corresponding coideal $\mathcal{O}(H\backslash G)$ is called a {\bf compact quasi-subgroup}. For the discussion on compact quasi-subgroups see \cite{AS23}. In particular, $\mathcal{O}(H\backslash G)$ is compact quasi-subgroup if and only if $\omega_H$ is positive.
Every quotient type coideal is a compact quasi-subgroup. Indeed, if $H\leq G$ then $h_H$ from subsection $2.1$ is clearly the Haar state of $H$ and we see that $\omega_H = h_H\circ q_H$ is an idempotent state.

Below we will present examples of coideals with relatively integrable coduals which are not, in general, compact quasi-subgroups. Such examples can be built from any non-Kac type CQG.

Let $G$ be a CQG and consider the CQG $G^{cop} \times G$ where $G^{cop}$ is $G$ equipped with its co-opposite quantum group structure (see, for example, \cite{FLS16}). Consider the quantum diagonal coideal
$$\mathcal{O}(D_G \backslash G^{cop}\times G) := \Delta_{G}(\mathcal{O}(G))\subseteq \mathcal{O}(G^{cop})\otimes \mathcal{O}(G).$$
It turns out that $l^\infty(\hat{D_G})$ is integrable if and only if $G$ is Kac type. Furthermore, $\mathcal{O}(D_G\backslash G^{cop}\times G)$ is of quotient type if and only if $G$ is the dual of a discrete group (cf. \cite{KS12, FLS16}). We will prove that $l^\infty(\hat{D_G})$ is always relatively integrable.

\begin{proposition} \label{diag}
    Let $G$ be a CQG. We have that $l^\infty(\hat{D_G})$ is $g$-integrable, where $\mathcal{O}(D_G\backslash G^{cop}\times G)$ is the quantum diagonal coideal of $G^{cop}\times G$ and
    $$\varphi_g = \epsilon_{G^{cop}}\otimes (\epsilon_G\circ \sigma^{G}_{i}).$$
\end{proposition}
\begin{proof}
    Let us recall from \cite{dCDT22} the modular automorphism groups $(\sigma^{G^{cop}\times G}_{t})_{t\in \mathbb{R}}$ and $(\sigma^{D_G\backslash G^{cop}\times G}_{t})_{t\in\mathbb{R}}$ on $\mathcal{O}(G^{cop}\times G)$ and $\mathcal{O}(D_G\backslash G^{cop}\times G)$ respectively, that are uniquely defined by the conditions
    $$h_{G^{cop}\times G}(ab) = h_{G^{cop}\times G}(b\sigma^{G^{cop}\times G}_{-i}(a)), ~a, b\in \mathcal{O}(G^{cop}\times G),$$
    and
    $$h_{G^{cop}\times G}(ab) = h_{G^{cop}\times G}(b\sigma_{-i}^{D_G\backslash G^{cop}\times G}(a)), ~ a,b\in \mathcal{O}(D_G\backslash G^{cop}\times G).$$

    We will use \cite[Corollary 1.12]{dCDT22} which states that $l^\infty(\hat D_G)$ is $g$-integrable if and only if $\epsilon_{G^{cop}\times G}\circ \sigma^{G^{cop}\times G}_{-i}\circ\sigma^{D_G\backslash G^{cop}\times G}_{i}$ extends to the character $\varphi_g$ on $\mathcal{O}(G^{cop}\times G)$.
    
    Accordingly, since $h_{G^{cop}\times G}|_{\mathcal{O}(D_G\backslash G^{cop}\times G)} = (h_{G^{cop}}\otimes h_G)|_{\Delta_G(\mathcal{O}(G))}$,
    $$\sigma_{i}^{D_G\backslash G^{cop}\times G} = \Delta_G(\sigma^G_{i}) = \sigma^{G^{cop}}_{i}\otimes \tau^G_{-i}.$$
    Therefore,
    \begin{align*}
        \epsilon_{G^{cop}\times G}\circ \sigma^{G^{cop}\times G}_{-i}\circ\sigma^{D_G\backslash G^{cop}\times G}_{i}  &= (\epsilon_{G^{cop}}\otimes \epsilon_G)\circ (\sigma^{G^{cop}}_{-i} \otimes \sigma^{G}_{-i})\circ (\sigma^{G^{cop}}_{i}\otimes \tau^G_{-i})
        \\
        &= \epsilon_{G^{cop}}\otimes (\epsilon_G\circ \sigma^G_{-i})
    \end{align*}
    which clearly defines a character on $\mathcal{O}(G)$.
\end{proof}
\begin{remark}
    In \cite[Remark 3.6]{FLS16} it was observed that $\omega_{D_G}(a\otimes b) = h_G(S(b)\otimes a)$. They also noted that each of the functionals $\omega_{D_G, t}$ that are defined by
    $$\omega_{D_G, t}(a\otimes b) = h_G(R_{G}(a)\sigma^G_{t + i/2}(b)), a\otimes b\in \mathcal{O}(G^{cop}\times G)$$
    are states. It turns out that $\omega_{D_G,0}\circ S^{-1}_{G^{cop}\times G} = \omega_{D_G}^g$. Indeed, it is straightforward to calculate that $\varphi_{g^{1/2}} = \epsilon_{G^{cop}}\otimes (\epsilon_G\circ \sigma^G_{-i/2})$. Furthermore, since $\varphi_{g^{1/2}}\circ \tau^{G^{cop}\times G}_{i/2} = \varphi_{g^{1/2}}$, $\varphi_{g^{-1/2}} = \varphi_{g^{1/2}}\circ R_{G^{cop}\times G}$. Then, for $a\otimes b\in \mathcal{O}(G^{cop}\otimes G)$,
\begin{align*}
    \omega^g_{D_G}\circ \tau_{-i/2}^{G^{cop}\times G}(a\otimes b)  &= \omega_{D_G}(g^{-1/2}\hat{\triangleright} R_{G^{cop}\times G}(a\otimes b))
    \\
    &= (\varphi_{g^{1/2}}*(w_{D_G}\circ R_{G^{cop}\times G}))(a\otimes b)
    \\
    &= ((\epsilon_{G^{cop}})\otimes (\epsilon_G\circ \sigma^G_{-i/2})\otimes \omega_{D_G})(a_{(1)}\otimes b_{(1)}\otimes R_{G^{cop}}(a_{(2)})\otimes R_G(b_{(2)}))
    \\
    &= \epsilon_G(\sigma^G_{-i/2}(b_{(1)}))h_G( (S_G\circ R_{G})(a)R_G(b_{(2)}))
    \\
    &= h_G( (S_{G}\circ R_{G})(a)\tau^G_{-i/2}(R_G(\sigma^G_{-i/2}(b))))
    \\
    &= h_G( \tau^{G}_{-i/2}( a  \sigma^G_{i/2}(R_G(b)) )  )
    \\
    &= h_G(a\sigma^G_{i/2}(R_G(b)))
    \\
    &= \omega_{D_G, 0}((R_{G_{cop}}\otimes R_{G})(a\otimes b)).
\end{align*}
    This shows $\omega^g_{D_G}\circ \tau_{-i/2}^{G^{cop}\times G} = \omega_{D_G,0}\circ R_{G^{cop}\times G}$ and hence the desired result.
\end{remark}
\end{subsection}

\begin{subsection}{$g$-Coamenability and Amenable Fusion Modules }
In this subsection, we introduce a notion of $g$-coamenability of the object $H$ corresponding to a coideal $\mathcal{O}(H\backslash G)$ and characterize it in terms of amenability of the fusion module $Rep(H)$.

Let $\pi : \mathcal{O}(G)\to \mathcal{B}(\mathcal{H})$ be a $*$-representation. We will write $C_\pi(G) = \overline{\pi(\mathcal{O}(G))}$. For example, $C_r(G) = C_{\lambda_G^g}(G)$ where $\lambda_G^g$ is the $g$-quasi-regular representation.

Now let $\rho$ be another $*$-representation of $\mathcal{O}(G)$. If $||\pi(a)||\leq ||\rho(a)||$ for all $a\in \mathcal{O}(G)$, then we write $\pi\prec \rho$ and say $\pi$ is weakly contained in $\rho$. If $\rho\prec \pi$ and $\pi \prec \rho$ then we say $\pi$ and $\rho$ are weakly equivalent.

Recall that $G$ is coamenable if and only if $\epsilon_G\prec \lambda_{G}$, i.e., the counit $\epsilon_G$ extends to $C_r(G)$. We will consider a generalization of this property for coideals with a relatively integrable codual.
\begin{definition}
    Let $G$ be a CQG and $\mathcal{O}(H\backslash G)$ a coideal where $l^\infty(\hat H)$ is $g$-integrable. We will say $H$ is $g$-coamenable if $\epsilon_G\prec \lambda_H^g$.
\end{definition}
In order to prove the main result of the paper which establishes a $C^*$-tensor categorical characterization of coamenability of a $*$-coalgebra associated with a coideal with a relatively integrable codual, we follow the strategy explained in the Introduction.

Accordingly, let us first establish the point (2) of this strategy. Let
$$\chi : \mathbb{C}[Rep(G)] \to \mathbb{C}, ~ u \mapsto \sum u_{i,i}, u\in Rep(G)$$
be the character map, which is a unital $*$-homomorphism. The proof of the following is an easy adaptation of \cite[Theorem 4.4]{K08}.
\begin{proposition}\label{Kestens Criterion for Quasi-Representations}
    Let $\pi : \mathcal{O}(G)\to B(H_\pi)$ be a $*$-representation. We have that $\epsilon_G\prec \pi$ if and only if Kesten's criterion holds for $\pi$, i.e., $n_u\in \sigma_{C_\pi(G)}( \Re(\pi(\chi(u))) )$ for every $u\in Rep(G)$. In particular, if $\mathcal{O}(H\backslash G)$ is a coideal where $l^\infty(\hat H)$ is $g$-integrable then $H$ is $g$-coamenable if and only if Kesten's criterion holds for $\lambda_H^g$.
\end{proposition}
What remains is establishing the unitary isomorphism alluded to at the start of this section. The main conceptual step is a generalization of the ``canonical trace'' on $\mathbb{C}[Rep(G)]$ and its GNS representation as it appeared in \cite{K08}. Loosely speaking, what we are doing is finding an analogue for an idempotent state of $G$ on the fusion ring induced by $Rep(G)$ in a way that generalizes how the ``canonical trace'' mentioned above is an analogue of the Haar state.

Let $\mathcal{O}(H\backslash G)$ be a coideal such that $l^\infty(\hat H)$ is $g$-integrable. Let $\tau_H^g = \omega_H^g\circ \chi$, which is a state on $\mathbb{C}[Rep(G)]$. Let $C^*_{\tau_H^g}(Rep(G))$ be the enveloping $C^*$-algebra of $\mathbb{C}[Rep(G)]$ with respect to the GNS space $L^2(Rep(G), \tau_H^g)$ of $\tau_H^g$. Let $\pi_{\tau_H^g}$ denote the corresponding GNS representation which satisfies $\pi_{\tau_H^g}(u)v = u\cdot v$ in $L^2(Rep(G), \tau_H^g)$.
\begin{proposition}\label{Quasi-Character Map Extends C*}
    Let $G$ be a CQG. Let $\mathcal{O}(H\backslash G)$ be a coideal with $g$-integrable codual. The map $\lambda_H^g\circ \chi : \mathbb{C}[Rep(G)] \to \mathcal{O}(G)$ extends to an isometric $*$-homomorphism $C^*_{\tau_H^g}(Rep(G))\to C_{\lambda_H^g}(G)$. There is also a unitary isomorphism $L^2(Rep(G), \tau_H^g)\cong L^2(\chi(\mathcal{O}(G)), \omega_H^g)$.
\end{proposition}
\begin{proof}
    The proof is an adaptation of the proof of \cite[Lemma 4.6]{K08}. Since $\omega_H^g(\chi(u)) = \tau_H^g(u)$, $\lambda_H^g\circ \chi$ extends to an isometric embedding
    $$L^2(Rep(G), \tau_H^g)\to l^2_g(\hat H) \cong L^2(\mathcal{O}(G), \omega_H^g).$$
    It then follows from a standard argument that the map
    $$\kappa : \lambda_H^g\circ \chi(\mathcal{O}(G)) \to  \pi_{\tau_H^g}(\mathbb{C}[Rep(G)]), \lambda_H^g(\chi(u))) \mapsto  \pi_{\tau_H^g}(u)$$
    extends to a $*$-isomorphism
    $$\kappa : \overline{\lambda_H^g\circ \chi(\mathcal{O}(G))}\to C^*_{\tau_H^g}(Rep(G)).$$
\end{proof}
For $\alpha\in Irr(H)$ let $P_\alpha \in l^\infty(\hat H)$ denote the orthogonal projection onto $H_\alpha$ which is nothing more than the identity operator in $M_{n_\alpha}\subset l^\infty(\hat H)$. It is easily observed that the map $\alpha\mapsto P_\alpha$ induces a unitary isomorphism
$$l^2(Irr(H)) \cong \overline{span\{\eta_{\hat H}(P_\alpha) : \alpha\in Irr(H)\}} \subset l^2_g(\hat H).$$
From this and the duality it is also easy to check that
$$L^2(\chi(\mathcal{O}(G)), \omega_H^g) = \overline{span\{\eta_{H}(q_H(\chi(u))) : u\in Rep(G)\}}\subset l^2_g(Irr(H))$$
as a closed subspace. For this last claim, consider how $(tr_u\otimes id)(W_H^g) = \lambda_H^g(\chi(u)^*)$ where $tr_u$ is the (non-normalized) trace on $M_{n_u}\subset l^\infty(\hat G)$, $u\in Rep(G)$. Altogether,
\begin{align}\label{Hilbert Space Inclusions}
    L^2(Rep(G), \tau_H^g)\cong L^2(\chi(\mathcal{O}(G)), \omega_H^g)\subset l^2(Irr(H))\subset l^2_g(\hat H).
\end{align}
The next lemma is the key step of our main result. It is what relates Kesten's criterion to the norms of the operators $\Gamma_u$ acting on $l^2(Irr(H))$.
\begin{lemma}\label{Kestens Criterion and Norm}
    For $a\in \mathbb{C}[Rep(G)]$, $||\lambda_H^g(\chi(a))|| = ||\Gamma_a||_{B(l^2(Irr(H))}$.
\end{lemma}
\begin{proof}
    For $u,v\in Rep(G)$
    $$\lambda_H^g(\chi(u))\eta_H(q_H(\chi(v))) = \eta_H(q_H(\chi(u\cdot v)))  = \Gamma_u(v\circ 1)$$
    and invoking the unitary isomorphism in Proposition~\ref{Quasi-Character Map Extends C*},
    $$\pi_{\tau_H^g}(u)v = \Gamma_u(v\circ 1).$$
    This shows that the restrictions of both $\Gamma_u$ and $\lambda_H^g(\chi(u))$ to $L^2(Rep(G), \tau_H^g)$ are equal to $\pi_{\tau_H^g}(u)$.

    Regarding \eqref{Hilbert Space Inclusions}, for $a\in \mathbb{C}[Rep(G)]$,
    $$||\pi_{\tau_H^g}(a)||_{B(L^2(Rep(G),\tau_H^g))}\leq ||\Gamma_a||_{B(l^2_g(Irr(H)))} || \leq  ||\lambda_H^g(\chi(a))||_{B(l^2_g(\hat H))}.$$
    However, thanks to Proposition~\ref{Quasi-Character Map Extends C*}, $||\lambda_H^g(\chi(a))|| = ||\pi_{\tau_H^g}(a)||$.
\end{proof}
Now we are able to prove the main result. Recall it:\\
~\\
{\bf Theorem~\ref{Coamenability and Amenable Fusion Modules}}
 Let $G$ be a CQG and let $\mathcal{O}(G/H)$ be a coideal such that $l^\infty(\hat H)$ is $g$-integrable and $H$ is Kac type. Then $\hat H$ is amenable if and only if $H$ is $g$-coamenable.


\begin{proof}

Assume $H$ is coamenable. For a contradiction, using Proposition~\ref{reform}, find $u = \sum_{v\in Irr(G)}\alpha_v v \in \mathbb{C}[Rep(G)]^+$ such that
\begin{align*}
    ||\Gamma_u||_{B(l^2(Irr(H)))} < d_G(u) = \sum_{v\in Irr(G)} \alpha_v n_v.
\end{align*}
Then $||\Gamma_{u + \overline{u}}|| < 2d_G(u)$.

Take $w\in Rep(G)$. Since $w\in M_{n_u}(\mathcal{O}(G))$ is unitary, $\chi(w)\in [-n_w, n_w]$ and therefore $||\lambda_H^g(\chi(w + \overline{w}))||\leq 2n_w$. We use this to deduce $||\lambda_H^g(\chi(u + \overline{u}))||\leq 2d_G(u)$. Then the same calculation with the counit as in the proof of Proposition~\ref{Kestens Criterion for Quasi-Representations} yields $||\lambda_H^g(\chi(u + \overline{u}))|| = 2d_G(u)$. So,
\begin{align*}
     2\sum_{v\in Irr(G)} \alpha_v n_v &= ||\lambda_H^g(\sum_{v\in Irr(H)}\alpha_v \chi(v + \overline{v}))||
     \\
     &= ||\Gamma_{u + \overline{u}}||_{B(l^2(Irr(H)))} ~ \text{(Lemma~\ref{Kestens Criterion and Norm})}
\end{align*}
which is the desired contradiction. Conversely, take $u\in Rep(G)$. By Lemma~\ref{Kestens Criterion and Norm}
$$||\chi(u + \overline{u})|| = ||\Gamma_{u + \overline{u}}|| = 2d_G(u) = 2n_u.$$
    
\end{proof}

\end{subsection}
\end{section}

\begin{section}{(Co)amenability and Duality}
\begin{subsection}{Amenability and Relative Amenability}
Fix a CQG $G$ and a coideal $\mathcal{O}(H\backslash G)\subseteq \mathcal{O}(G)$. Here we will introduce a notion of amenability of $\hat H$ which is meant to be dual to coamenability of $H$ in a sense that is analogous to the fact that amenability of $\hat G$ is equivalent to coamenability of $G$ \cite{T06}. One benefit of this notion of amenability is that it does not require $g$-integrability of $l^\infty(\hat H)$.

A DQG $\hat G$ is amenable if $l^\infty(\hat G)$ admits a $\hat G$-invariant state, i.e., a state $m : l^\infty(\hat G)\to \mathbb{C}$ such that
$$m(x\triangleleft a) = \epsilon_G(a) m(x), ~ x\in l^\infty(\hat G), a\in \mathcal{O}(G).$$
Amenability is a quantum analogue of the classical notion amenability of a discrete group. In the setting of discrete groups, amenability has a well-known generalization to quotients of groups by their subgroups and enjoys interesting connections to certain properties of their associated operator algebras. To name a couple, coamenable quotients have characterizations in terms of coamenable inclusions of von Neumann algebras \cite{MP03} and the structure of the associated $FH$-boundaries constructed in \cite{BK21}. We will present an analogue of this property for relatively integrable coideals.

In what follows, recall the notation established in the paragraphs following Proposition~\ref{Plancherel Theorem}.
\begin{definition}\label{Coamenable Inclusion}
    Let $\mathcal{O}(H\backslash G)\subseteq \mathcal{O}(G)$ be a coideal. We will say $\hat H$ is {\bf amenable} if $l^\infty(\hat H)$ admits a $\hat G$-invariant state.
\end{definition}
\begin{remark}\label{gAmenable Remark}
    Recall our notation for the coaction induced by $\lambda^g_H$ in Section $4.1$. We will sometimes work instead with functionals in $l^1(\hat G)$ and will use the notation
    $$f* m(x) := (f\otimes m)((W_H)^*(1\otimes x)W_H), ~ x\in l^\infty(\hat H), f\in l^1(\hat G).$$
    We will also use the notation
    $$f*_g m(x) := (f\otimes m)((W_H^g)^*(1\otimes x)W_H^g), ~ x\in l^\infty(\hat H), f\in l^1(\hat G).$$
    Fix a state $m\in S(l^\infty(\hat H))$. Thanks to Proposition~\ref{Coaction on Hhat}, we can see that $f*m = f(1)m$ if and only $f*_g m = f(1)m$.
\end{remark}
For the rest of this subsection, we will establish some basic results on amenability.

First, we consider a notion of amenability considered recently for coideals in \cite{KKSV22, AS23} that is in a sense ``opposite'' to amenability of $\hat H$.
\begin{definition}
    Let $\mathcal{O}(H\backslash G)$ be a coideal such that $l^\infty(\hat H)$ is $g$-integrable. We say a unital completely positive map $\Psi : l^\infty(\hat G)\to l^\infty(\hat G)$ is {\bf $\hat G$-equivariant} if
    $$\Psi(x\triangleleft a) = \Psi(x)\triangleleft a, ~ a\in \mathcal{O}(G), x\in l^\infty(\hat G).$$
    We say $l^\infty(\hat H)$ is {\bf relatively amenable} if there exists a $\hat G$-equivariant unital completely positive map $l^\infty(\hat G)\to l^\infty(\hat H)$.
\end{definition}
\begin{proposition}\label{Hereditary Result for Amenability}
    Let $G$ be a CQG and $\mathcal{O}(H\backslash G)$ a coideal. We have that $\hat G$ is amenable if and only if $l^\infty(\hat H)$ is relatively amenable and $\hat H$ is amenable.
\end{proposition}
\begin{proof}
    Assume $\hat G$ is amenable. It is straightforward to show that $m : l^\infty(\hat G)\to \mathbb{C}\subseteq l^\infty(\hat H)$ satisfies the requirement for both amenability of $\hat H$ and relative amenability. Conversely, if $m : l^\infty(\hat H)\to \mathbb{C}$ is a $\hat G$-invariant state and $\Psi : l^\infty(\hat G)\to l^\infty(\hat H)$ is $\hat G$-equivariant then it is straightfoward to calculate that $m\circ \Psi : l^\infty(\hat G)\to \mathbb{C}$ is a $\hat G$-invariant state.
\end{proof}
Next, we have a partial generalization of the characterization found in \cite{MP03}for the duals of discrete groups.

Let $L^\infty(G) = \mathcal{O}(G)''\subseteq \mathcal{B}(l^2(\hat G))$ and
$$L^\infty(H\backslash G) = \mathcal{O}(H\backslash G)''\subseteq L^\infty(G).$$
Let $M$ be a von Neumann algebra with a normal faithful state $h$. Given an inclusion of von Neumann algebras $N\subseteq M$, we denote $\langle N, M\rangle = JNJ'\cap \mathcal{B}(L^2(M,h))$ where $J$ is the modular conjugation in the standard form of $M$ in $\mathcal{B}(L^2(M, h))$.
\begin{definition}
    We say the inclusion $N\subseteq M$ is coamenable if there exists a conditional expectation $\langle N,M\rangle \to M$ onto $M$.
\end{definition}
Let $G = \hat\Gamma$ be the dual of a discrete group $\Gamma$ and $\Lambda\leq \Gamma$ a subgroup. Popa and Monod established in \cite{MP03} that the inclusion $L^\infty(\hat\Lambda)\subseteq L^\infty(\hat\Gamma)$ is coamenable if and only if $\Gamma / \Lambda$ is amenable. We have been unable to prove a full analogue of this result for quantum groups. For the result we do have, we are forced to consider the following class of CQGs recently developed by Crann\cite{C19}.
\begin{definition}
    Let $G$ be a CQG. We say $\hat G$ is {\bf inner amenable} if $L^\infty(G)$ admits a $\hat G$-invariant state.
\end{definition}
\begin{remark}
    Every Kac type CQG is inner amenable. Indeed, the Haar state $h_G\in L^\infty(G)^*$ is $\hat G$-invariant if and only if $G$ is Kac type (cf. \cite{KKSV22, I02}). Note, also, that if $L^\infty(G)$ admits a $\hat G$-invariant state $\tau : L^\infty(G) \to \mathbb{C}$, then the restriction $\tau|_{C_r(G)} : C_r(G)\to \mathbb{C}$ is a $\hat G$-invariant state which implies the existence of a tracial state on $C_r(G)$ (cf. \cite{AS22(1)}).

    It is also the case that every coamenable CQG is inner amenable. For this, it is easy to show the counit is $\hat G$-invariant.
\end{remark}
\begin{proposition}
    Let $\mathcal{O}(H\backslash G)$ be a coideal. Assume that $\hat G$ is inner amenable. If $L^\infty(H\backslash G)\subseteq L^\infty(G)$ is coamenable (as an inclusion of von Neumann algebras) then $\hat H$ is amenable.
\end{proposition}
\begin{proof}
    Recall that the extension of the unitary antipode $R_{\hat G}$ on $l^\infty(\hat G)$ to $\mathcal{B}(l^2(\hat G))$ is given by $R_{\hat G}(T) = J_G T^* J_G$ where $J_G$ is the modular conjugation for $L^\infty(G)$. Furthermore, it is known that
    $$L^\infty(H\backslash G)'\cap l^\infty(\hat G) = R_{\hat G}(l^\infty(H)) ~ \text{\cite{ILP98}}$$
    (it should be noted that in \cite{ILP98}, they make use of the right regular representation instead). Then, since $R_{\hat G}^2 = id$,
    $$l^\infty(\hat H)\subseteq J_GL^\infty(H\backslash G)J_G' \cap \mathcal{B}(l^2(\hat G)).$$
    Now, let $\tau : L^\infty(G)\to \mathbb{C}$ be a $\hat G$-invariant state and $E : \langle L^\infty(H\backslash G), L^\infty(G)\rangle \to L^\infty(G)$ the given conditional expectation. Given $x\in l^\infty(\hat H)$
    \begin{align*}
        (id\otimes m\circ E|_{l^\infty(H)})(W_{\hat G}^*(1\otimes x)W_{\hat G}) &= (id\otimes m)(W_{\hat G}^*(1\otimes E(x))W_{\hat G}) = m(E(x))
    \end{align*}
    where in the first equation we are using $L^\infty(G)$-bimodularity of $E$. This shows $m\circ E|_{l^\infty(\hat H)}$ is $\hat G$-invariant as desired.
\end{proof}
At this time, we are unable to establish a converse of the above proposition.

\end{subsection}

\begin{subsection}{The Kac Property of Coideals}
In these remaining subsections we characterize coamenability of $H$ with amenability of $\hat H$ under an assumption on $H$ we call the Kac property. For the Kac property, we will require $l^\infty(\hat H)$ to be relatively integrable.

This duality result generalizes the main theorem in \cite{T06} as well as gives us an analogue of Theorem~\ref{Coamenability and Amenable Fusion Modules} for amenability. While in the statements of Theorem~\ref{Duality for Coamenability of Quasi Regular Representations} and Corollary~\ref{Hereditary Result for Coamenability}, we assume the Kac property of $H$, in reality we only require this assumption for one direction of the asserted claims. We explain this precisely in Remarks~\ref{Remark on Duality Result} and~\ref{Remark on Hereditary Result}. So, some of our results are also partial progress for arbitrary coideals.

Let $G$ be a CQG and $\mathcal{O}(H\backslash G)$ a coideal. Let us begin by introducing the Kac property for $H$. For this, we require a discussion on some of the additional structure of $l^\infty(\hat H)$.

Let $\psi_{\hat H} : c_{00}(\hat H) \to \mathbb{C}$ be a $g$-invariant integral, which is normalized so that $\psi_{\hat H}(P_H\triangleleft a) = g(a)$. Now, let $\delta_{g\hat H} = (\mu_{\alpha})_{\alpha\in Irr(H)}$, where $\mu_\alpha\in M_{n_\alpha}$ are positive invertible matrices chosen such that
$$\kappa_{t}^{\hat H}(x) = \delta_{g\hat H}^{-it/2}x\delta_{g\hat H}^{it/2} = (g\delta_{\hat G}^{1/2})^{-it} x (g\delta_{\hat G}^{1/2})^{it}, ~ x\in l^\infty(\hat H)$$
where $\delta_{\hat G} = (Q_u^{-2})_{u\in Irr(G)}$. 

\begin{definition}\label{Kac Type Definition}
    Let $G$ be a CQG and $\mathcal{O}(H\backslash G)$ a coideal where $l^\infty(\hat H)$ is $g$-invariant. We will say $H$ is {\bf Kac type} if $\kappa_t^{\hat H} = id$ for every $t\in \mathbb{R}$.
\end{definition}
As observed, for example, in the proof of \cite[Theorem 1.11]{dCDT22},
$$\psi_{\hat H}(x) = \sum_{\alpha\in Irr(H)} c_\alpha tr(\mu_\alpha x), x\in c_{00}(\hat H)$$
for some fixed set of scalars $c_\alpha > 0$. We will normalize the $\mu_\alpha$'s such that $tr(\mu_\alpha) = tr(\mu_\alpha^{-1})$. 

We have the following analogues of properties that characterize the Kac property for quantum groups. The proof is quite obvious from the definitions.
\begin{proposition}\label{Kac type Properties}
    Let $\mathcal{O}(H\backslash G)$ be a coideal. The following are equivalent:
    \begin{enumerate}
        \item $H$ is Kac type.
        \item $\mu_\alpha  = I_{n_\alpha}$ for all $\alpha\in Irr(H)$.
        \item $\psi_{\hat H}$ is tracial.
    \end{enumerate}
\end{proposition}
\begin{example}\label{Kac Examples}
\cite[Corollary 1.17]{dCDT22} shows that  if $l^\infty(\hat H)$ is $\delta^{-1/2}_{\hat G}$-integrable, then $H$ is Kac type. Such examples occur when $\mathcal{O}(H\backslash G)\subseteq \mathcal{O}(G)$ is a so-called co-Gelfand pair (cf. \cite{dCDT22}). Certain non-standard Podle\'s spheres $\mathcal{O}(S^2_{q,t})\subseteq \mathcal{O}(SU_q(2))$ are  co-Gelfand pairs (cf. \cite[Section 3]{dCDT22}).
\end{example}

\end{subsection}
\begin{subsection}{A Duality Result}
In this subsection, we will prove a generalization of Tomatsu's theorem \cite{T06} for the quantum quotients of CQGs under the assumption that the quantum quotients are Kac type. We emphasize that this result fundamentally depends on the work in Section $4.1$ which establishes a Plancherel's theorem for relatively integrable coideals of $\hat{G}$ that, in particular, yields that $\ell^\infty(\hat{H})$ and $C^{\lambda^g_H}(G)$ are represented on the same Hilbert space.

In what follows, for every $u\in Rep(G)$ we choose a basis and matrix units $E_{i,j}^u\in M_{n_u}$. We let $\delta^u_{i,j}$ denote the corresponding dual basis vectors in $(M_{n_u})_*$. By definition, they satisfy $\delta_{i,j}^u(E_{k,l}^v) = \delta_{i,k}\delta_{j,l}\delta_{u,v}$.

Thanks to Proposition~\ref{Kac type Properties}, when $H$ is Kac type $\psi_{\hat H}$ is tracial. Hence, it is straightforward to see that the representation $l^\infty(\hat H) \subseteq B(l^2_g(\hat H))$ is a standard form of $l^\infty(\hat H)$ such that the positive cone $l^2_g(\hat H)^+$ is the linear span of positive matrices. The following lemma is standard, hence we omit its proof.
\begin{lemma}\label{Modular Lemma Kac}
    Assume $H$ is Kac type. Let $\xi\in l^2_g(\hat H)^+$ and $u\in Rep(G)$. Then $1\otimes \xi$ and $[\lambda_H(\delta^u_{i,j})\xi]_{j,i}$ lie in the positive cone $(L^2(M_{n_u})\otimes l^2_g(\hat H))^+$.
\end{lemma}
~\\
{\bf Theorem~\ref{Duality for Coamenability of Quasi Regular Representations}} Let $G$ be a CQG and let $\mathcal{O}(G/H)$ be a coideal such that $l^\infty(\hat H)$ is $g$-integrable and $H$ is Kac type. Then $\hat H$ is amenable if and only if $H$ is $g$-coamenable.
\begin{proof}
$(\implies)$ The calculations are the same as the calculations done by Blanchard and Vaes \cite{BV02}. We will repeat them here since their article is not widely available. Let $m : l^\infty(\hat H)\to \mathbb{C}$ be a $\hat G$-invariant state on $l^\infty(\hat H)$. Recall the notation from Remark~\ref{gAmenable Remark}. Using the fact $l^\infty(\hat H) \subseteq B(l^2_g(\hat H))$ is in standard form, find a net of unit vectors $(\xi_\alpha)\subseteq l^2_g(\hat H)$ such that $w_{\xi_\alpha}$ weak$^*$ approximates $m|_{l^\infty(\hat H)}$. A standard convexity argument (Day's trick) allows us to conclude that
$$||f*_g w_{\xi_\alpha} - f(1)w_{\xi_\alpha}||\to 1,~ f\in l^1(\hat G).$$
Next, fix $u\in \mathrm{Rep}(G)$ and define the functionals $\eta_\alpha, \mu_\alpha\in M_{n_u}(l^\infty(\hat H))$ by setting
\begin{align}\label{Functional 1}
    \eta_\alpha(x) &= (\mathrm{tr}_u\otimes w_{\xi_\alpha})(x), x\in M_{n_u}(l^\infty(\hat H))
\end{align}
and
\begin{align}\label{Functional 2}
    \mu_\alpha(x) &= \sum_{i,j=1}^{n_u} \delta_{i,j}^u*_g w_{\xi_\alpha}(x_{i,j}), x = [x_{i,j}]\in M_{n_u}(l^\infty(\hat H))
\end{align}
It is clear that $\eta_\alpha$ is positive. We have that $\mu_\alpha$ is positive because it can be checked that $\mu_\alpha(x) = (\mathrm{tr}\otimes w_{\xi_\alpha})([\lambda_H^g(\delta_{k,l}^u)]_{l,k}^* x [\lambda_H^g(\delta_{k,l}^u)]_{l,k})$. Then
\begin{align*}
    (\eta_\alpha - \mu_\alpha)([x_{i,j}]) &= \sum_{i,j=1}^{n_u} \delta^u_{i,j}(1) w_{\xi_\alpha}(x_{i,j}) - \delta_{i,j}^u*_g w_{\xi_\alpha}(x_{i,j}) \to 0.
\end{align*}
Again, a standard convexity argument implies that $||\mu_\alpha - \eta_\alpha||_{M_{n_u}(l^\infty(\hat H))_*} \to 0$. By Lemma~\ref{Modular Lemma Kac},
\begin{align*}
    1\otimes \xi_\alpha, ~ [\lambda_H^g(\delta^u_{i,j})\xi_\alpha]_{j,i} \in (L^2(M_{n_u})\otimes l^2_g(\hat H))^+.
\end{align*}
It is obvious $1\otimes\xi_\alpha$ implements $\eta_\alpha$ as a vector state. It is easily checked that $[\lambda_H^g(\delta^u_{i,j})\xi_\alpha]_{j,i}$ implements $\mu_\alpha$ as a vector state. Recall the Powers-St\o rmer inequality states that $||\xi - \eta|| \leq ||w_\xi - w_\eta||$, where there is a von Neumann algebra $M$ represented in a standard form on a Hilbert space $H$, and $\xi,\eta\in H$ lie in the positive cone of $H$ (see \cite{H75}). Combining our calculations with an application of the Powers-St\o rmers inequality gives


\begin{align*}
    &||1\otimes \xi_\alpha - [\lambda_H^g(\delta_{k,l}^u)\xi_\alpha]_{l,k}||_{L^2(M_{n_u})\otimes l^2_g(\hat H)} \leq ||\mu_\alpha - \eta_\alpha ||_{M_{n_u}(l^\infty(\hat H))_*} \to 0.
\end{align*}
Then, for $f = \sum \alpha_{i,j}\delta_{i,j}^u\in (M_{n_u})\subseteq l^1(\hat G)$,
\begin{align*}
    &||\lambda_H^g(f)\xi_\alpha - f(1)\xi_\alpha||_{L^2(M_{n_u})\otimes l^2_g(\hat H)}
    \\
    &= ||(f^t\otimes id)(1_{n_u}\otimes \xi_\alpha - [\lambda_H^g(\delta_{k,l}^u)\xi_\alpha]_{l,k})||_{L^2(M_{n_u})\otimes l^2_g(\hat H)} \to 0
\end{align*}where $f^t = \sum \alpha_{i,j}\delta^u_{j,i}$. Hence we conclude that $||\lambda_{H}^g(a)\xi_\alpha - \epsilon_G(a)\xi_\alpha||_2\to 0$ for $a\in \mathcal{O}(G)$.
This implies $g$-coamenability of $H$.

$(\impliedby)$ If $\epsilon_G\prec \lambda_H$, then $\epsilon_G$ extends to a character $\epsilon_G : C_{\lambda_H}(\hat G) \to \mathbb{C}$. Using injectivity of $\mathbb{C}$, obtain a unital completely positive extension $\varphi : B(l^2_g(\hat H))\to \mathbb{C}$. Then
\begin{align*}
    (id\otimes \varphi)((W_H^g)^*(1\otimes T)W_H^g) &= (id\otimes \varphi)((id\otimes \lambda_H^g)(\mathbb{W}_{\hat G})^*(1\otimes T)(id\otimes \lambda_H^g)(\mathbb{W}_{\hat G}))
    \\
    &= \varphi(T)
\end{align*}
where in the second last equation we used the fact $C_{\lambda_H}(\hat G)$ lies in the multiplicative domain of $\varphi$. Therefore, $m := \varphi|_{l^\infty(\hat H)}$ is $\hat G$-invariant.
\end{proof}
\begin{remark}\label{Remark on Duality Result}
    The converse of Theorem~\ref{Duality for Coamenability of Quasi Regular Representations} did require the Kac property of $H$. In particular, the following statement holds for arbitrary CQGs: let $G$ be a CQG and $\mathcal{O}(H\backslash G)$ a coideal where $l^\infty(\hat H)$ is $g$-integrable. If $H$ is $g$-coamenable then $\hat H$ is amenable.
\end{remark}

\subsection{Applications of Duality to Hereditary Properties}
Here we apply the duality result from the previous subsection to obtain hereditary results for ($g$-co)amenability. We also relate $g$-coamenability of a quantum quotient with another notion of $g$-coamenability of a coideal which has appeared previously in the context of compact quasi-subgroups.

Making use Theorem~\ref{Duality for Coamenability of Quasi Regular Representations} in combination with Theorem~\ref{Coamenability and Amenable Fusion Modules} we immediately obtain the following.
\begin{corollary}\label{Coamenable Inclusions and Amenable Fusion Modules}
    Let $G$ be a CQG and $\mathcal{O}(H\backslash G)$ a coideal such that $l^\infty(\hat H)$ is relatively integrable and $H$ is Kac type. We have that $\hat H$ is amenable if and only if $((\mathbb{C}[Rep(H)], Irr(H), c^\beta_{u,\alpha}), d^H)$ is amenable as a fusion module over $((\mathbb{C}[Rep(G)], Irr(G), N^w_{u,v}), d^G)$.
\end{corollary}
We will spend the rest of this subsection using Theorem~\ref{Duality for Coamenability of Quasi Regular Representations} to study another notion of coamenability of coideals that has recently appeared in the literature.

In what follows, we will denote
$$C_r(H\backslash G) := \overline{\mathcal{O}(H\backslash G)}\subseteq C_r(G).$$
\begin{definition}\label{Coamenable Coideal}
    Let $\mathcal{O}(H\backslash G)$ be a coideal such that $l^\infty(\hat H)$ is $g$-integrable. We will say $\mathcal{O}(H\backslash G)$ is $g$-coamenable if $\lambda^g_H\prec \lambda_G$.
\end{definition}
\begin{remark}\label{Coamenability Remark}
    A notion of coamenability of a quotient type coideal was introduced in \cite{KKSV22} and later studied in \cite{AS23} for compact quasi-subgroups. Their definition is as follows: a coideal $\mathcal{O}(H\backslash G)$ is coamenable if the restriction of the counit $\epsilon_G|_{\mathcal{O}(H\backslash G)}$ extends to a state on $C_r(H\backslash G)$.
    
    It was shown by the first author in \cite{AS23} that a compact quasi-subgroup $\mathcal{O}(H\backslash G)$ is coamenable in the sense of above if and only if it is $1$-coamenable. This result was first obtained for quotient type coideals in \cite{KKSV22}.

    Currently, we do not know if $g$-coamenabilty of $\mathcal{O}(H\backslash G)$ is equivalent to coamenability of $\mathcal{O}(H\backslash G)$ in general.
\end{remark}
With such amenability properties at hand, we find an analogue of Proposition~\ref{Hereditary Result for Amenability}.
\begin{corollary}\label{Hereditary Result for Coamenability}
    Let $G$ be a CQG and $\mathcal{O}(H\backslash G)$ a coideal such that $l^\infty(\hat H)$ is $g$-integrable. Suppose $H$ is Kac type. We have that $G$ is coamenable if and only if $H$ is $g$-coamenable and $\mathcal{O}(H\backslash G)$ is $g$-coamenable.
\end{corollary}
\begin{proof}
    Assume $G$ is coamenable. It is well-known that $G$ is coamenable if and only if $\lambda_G\cong \varpi_G$, where $\varpi_G$ is the universal representation on $\mathcal{O}(G)$. Then $\lambda_H^g\prec \varpi_G\cong \lambda_G$ by the universal property. Next, by Tomatsu's theorem, $\hat G$ is amenable, and then we deduce amenability of $\hat H$ from Proposition~\ref{Hereditary Result for Amenability}. From Theorem~\ref{Duality for Coamenability of Quasi Regular Representations} we deduce $g$-coamenability of $H$. Conversely, if $H$ is $g$-coamenable and $\mathcal{O}(H\backslash G)$ is $g$-coamenable then $\epsilon_G\prec \lambda_H^g\prec \lambda_G$.
\end{proof}
\begin{remark}\label{Remark on Hereditary Result}
    In the proof of Proposition~\ref{Hereditary Result for Coamenability}, the Kac property was only required for the following statement: if $G$ is coamenable then $H$ is coamenable.

    In particular, given a CQG $G$ and an arbitrary coideal $\mathcal{O}(H\backslash G)$ such that $l^\infty(\hat H)$ is $g$-integrable, it is true that if $G$ is coamenable then $\mathcal{O}(H\backslash G)$ is $g$-coamenable.

    Likewise, the converse of Corollary~\ref{Hereditary Result for Coamenability} is true without the Kac assumption. To put it precisely, the following statement holds for arbitrary CQGs: if $H$ is $g$-coamenable and $\mathcal{O}(H\backslash G)$ is $g$-coamenable then $G$ is coamenable.
\end{remark}
\begin{example}\label{Podles Spheres are Amenable}
    Consider the case where $G = SU_q(2)$. Let $\mathcal{O}(S^2_{q,t})  = $\newline $\mathcal{O}(B_t\backslash SU_q(2))$ be one of Podle\'s spheres of $SU_q(2)$ as in \cite[Section 3]{dCDT22}. As mentioned in Example~\ref{Kac Examples}, $B_t$ is Kac type and $l^\infty(\hat B_t)$ is $\delta_{SU_q(2)}^{-1/2}$-integrable. Since $SU_q(2)$ is coamenable, we immediately obtain from Corollary~\ref{Hereditary Result for Coamenability} that $B_t$ is $\delta_{SU_q(2)}^{-1/2}$-coamenable. 
\end{example}
A dual property to ($g$-)coamenability (where coamenability without $g$- refers to the definition of coamenability in Remark~\ref{Coamenability Remark}) is relative amenability of coideals of $l^\infty(\hat G)$ (cf. Section $4.1$).  In the case where $\mathcal{O}(H\backslash G) = \mathcal{O}(G)$, ($g$-)coamenability of $\mathcal{O}(G)$ is equivalent to coamenability of $G$. Note that in this case $H = \{e\}$. On the other hand, relative amenability of $\mathbb{C} = l^\infty(\{e\})$ is equivalent to amenability of $G$. In this case, the equivalence between $g$-coamenability and relative amenability is due to the celebrated theorem of Tomatsu \cite{T06}. This duality holds when $l^\infty(\hat H)$ is of quotient type (see \cite{KKSV22, AS23}), however, it remains open in general to determine if relative amenability of $l^\infty(\hat H)$ is equivalent to ($g$-)coamenability of $\mathcal{O}(H\backslash G)$. We emphasize that {\it neither} direction has been established in the literature in general.

The combination of Proposition~\ref{Hereditary Result for Amenability} and Corollary~\ref{Hereditary Result for Coamenability} allows us to establish duality between $g$-coamenability of $\mathcal{O}(H\backslash G)$ and relative amenability of $l^\infty(\hat H)$ when $H$ is $g$-coamenable. We emaphasize that the Kac property is not required for this result.
\begin{corollary}\label{Duality Result for Relative Amenability and Coamenability}
    Let $G$ be a CQG and $\mathcal{O}(H\backslash G)$ a coideal such that $l^\infty(\hat H)$ is $g$-integrable. Suppose $H$ is $g$-coamenable. We have that $l^\infty(\hat H)$ is relatively amenable if and only if $\mathcal{O}(H\backslash G)$ is $g$-coamenable.
\end{corollary}
\begin{proof}
    Assume $l^\infty(\hat H)$ is relatively amenable. Thanks to the proof of Theorem~\ref{Duality for Coamenability of Quasi Regular Representations} (cf. Remark~\ref{Remark on Duality Result}), $\hat H$ is amenable. By Proposition~\ref{Hereditary Result for Amenability}, $\hat G$ is amenable. Using Tomatsu's theorem \cite{T06}, $G$ is coamenable. Then $\mathcal{O}(H\backslash G)$ is coamenable by the proof of Corollary~\ref{Hereditary Result for Coamenability} (cf. Remark~\ref{Remark on Hereditary Result}).

    Conversely, if $\mathcal{O}(H\backslash G)$ and $H$ are both $g$-coamenable then $G$ is coamenable by the proof of Corollary~\ref{Hereditary Result for Coamenability} (cf. Remark~\ref{Remark on Hereditary Result}). By Tomatsu's theorem, $\hat G$ is amenable. Then $l^\infty(\hat H)$ is relatively amenable by Proposition~\ref{Hereditary Result for Amenability}.
\end{proof}
\end{subsection}
\end{section}

{\small\bibliography{Amenable_TCs}}
\bibliographystyle{alpha}

\end{document}